\documentclass[12pt]{amsart}

\usepackage{amssymb}
\usepackage{latexsym}
\usepackage{exscale}
\usepackage{lscape}

\usepackage[colorlinks=true, pdfstartview=FitV, linkcolor=blue,
citecolor=blue, urlcolor=blue]{hyperref}

\newtheorem{theorem}{Theorem}[section]
\newtheorem{lemma}[theorem]{Lemma}

\theoremstyle{definition}

\theoremstyle{remark}
\newtheorem{remark}[theorem]{Remark}

\numberwithin{equation}{section}

\renewcommand{\emptyset}{\mbox{\textup{\O}}}

\newcommand{\re}{\mathbb R}

\DeclareMathOperator{\supp}{supp}

\newcommand{\subRn}{{{\mathbb R}^n}}

\newcommand{\la}{\langle}
\newcommand{\ra}{\rangle}

\newcommand{\M}{\overline{M}}

\newcommand{\Q}{\mathcal{Q}}

\def\Xint#1{\mathchoice
  {\XXint\displaystyle\textstyle{#1}}%
  {\XXint\textstyle\scriptstyle{#1}}%
  {\XXint\scriptstyle\scriptscriptstyle{#1}}%
  {\XXint\scriptscriptstyle\scriptscriptstyle{#1}}%
  \!\int}
\def\XXint#1#2#3{{\setbox0=\hbox{$#1{#2#3}{\int}$}
    \vcenter{\hbox{$#2#3$}}\kern-.5\wd0}}

\def\avgint{\Xint-}

\allowdisplaybreaks

\begin{document}

\title[Sharp weighted estimates]{Sharp weighted estimates for classical operators}

\author{David Cruz-Uribe, SFO}
\address{David Cruz-Uribe, SFO\\
Dept. of Mathematics \\ Trinity College \\
Hartford, CT 06106-3100, USA} \email{david.cruzuribe@trincoll.edu}

\author{Jos\'e Mar{\'\i}a Martell}

\address{Jos\'e Mar{\'\i}a Martell
\\
Instituto de Ciencias Matem\'aticas CSIC-UAM-UC3M-UCM
\\
Consejo Superior de Investigaciones Cient{\'\i}ficas
\\
C/ Nicol\'as Cabrera, 13-15
\\
E-28049 Madrid, Spain} \email{chema.martell@icmat.es}

\author{Carlos P\'erez}
\address{Carlos P\'erez
\\
Departamento de An\'alisis Matem\'atico, Facultad de Matem\'aticas\\
Universidad de Sevilla, 41080 Sevilla, Spain}
\email{carlosperez@us.es}


\thanks{The first author was supported by a grant from the Faculty
  Research Committee and the Stewart-Dorwart Faculty Development Fund
  at Trinity College; the first and third authors are supported by
  grant MTM2009-08934 from the Spanish Ministry of Science and
  Innovation; the second author is supported by grant MTM2007-60952
  from the Spanish Ministry of Science and Innovation and by CSIC PIE
  200850I015. }

\begin{abstract}
We give a general method based on dyadic    Calder\'on-Zygmund theory to prove sharp one and two-weight norm inequalities for some of the classical operators of harmonic analysis: the Hilbert and Riesz transforms, the Beurling-Ahlfors operator, the maximal singular integrals associated to these operators, the dyadic square function and the vector-valued maximal operator.

In the one-weight case we prove the sharp dependence on the $A_p$
constant by finding the best value for the exponent $\alpha(p)$ such
that
\[ \|Tf\|_{L^p(w)} \leq
C_{n,T}\,[w]_{A_p}^{\alpha(p)}\|f\|_{L^p(w)}.
\]
For the Hilbert transform, the Riesz transforms and the
Beurling-Ahlfors operator the sharp value of $\alpha(p)$ was found by
Petermichl and Volberg~\cite{MR2354322,MR2367098,MR1894362}; their
proofs used approximations by the dyadic Haar shift operators, Bellman
function techniques, and two-weight norm inequalities.
Our proofs again depend on dyadic approximation, but avoid Bellman functions and two-weight norm inequalities.  We
instead use a recent result due to A.~Lerner~\cite{lernerP2009} to estimate the oscillation of dyadic operators. By applying this we
get a straightforward proof of the sharp dependence on the $A_p$ constant for any operator that can be approximated by Haar shift
operators. In particular, we provide a unified approach for the Hilbert and Riesz transforms, the Beurling-Ahlfors operator (and
their corresponding maximal singular integrals),  dyadic paraproducts and Haar multipliers. Furthermore, we completely solve the open problem
of sharp dependence for the dyadic square functions and vector-valued Hardy-Littlewood maximal function.

In the two-weight case we use the very same techniques to prove sharp
results in the scale of $A_p$ bump conditions.  For the singular
integrals considered above, we show they map $L^p(v)$ into $L^p(u)$,
$1<p<\infty$, if the pair $(u,v)$ satisfies
\[ \sup_Q \|u^{1/p}\|_{A,Q}\|v^{-1/p}\|_{B,Q} < \infty, \]
where $\bar{A}\in B_{p'}$ and $\bar{B}\in B_p$ are Orlicz functions.  This condition is sharp.  Furthermore, this
condition characterizes (in the scale of these $A_p$ bump conditions) the corresponding two-weight norm inequality for the
Hardy-Littlewood maximal operator $M$ and its dual: i.e., $M:L^p(v)\longrightarrow L^p(u)$ and $M:L^{p'}(u^{1-p'})\longrightarrow L^p(v^{1-p'})$. Muckenhoupt and Wheeden conjectured that these two inequalities for $M$ are sufficient for the Hilbert transform to be bounded from $L^p(v)$
into $L^p(u)$. Thus, in the scale of $A_p$ bump conditions, we prove their conjecture. We prove similar, sharp two-weight results for
the dyadic square function and the vector-valued maximal operator.

\end{abstract}

\date{\today}

\maketitle

\section{Introduction}

The problem of proving one and two-weight norm inequalities for the
classical operators of harmonic analysis---singular integrals,
square functions, maximal operators---has a long and complex
history.  In the one weight case, the (nearly) universal sufficient
and (often) necessary condition for an operator to be bounded on
$L^p(w)$  is the $A_p$ condition:  given $1<p<\infty$, a weight $w$
(i.e., a non-negative, locally integrable function) is in $A_p$ if
\[ [w]_{A_p} = \sup_Q \left(\avgint_Q w(x)\,dx\right)\left(\avgint_Q w(x)^{1-p'}\,dx\right)^{p-1} < \infty, \]
where the supremum is taken over all cubes in $\re^n$ and $\avgint_Q
w(x)\,dx = |Q|^{-1}\int_Q w(x)\,dx$.  For more on one-weight
inequalities we refer the reader to
\cite{MR1800316,MR807149,MR2463316}.

An important question is to determine the best constant in terms of
the $A_p$ constant $[w]_{A_p}$.  More precisely, given an operator
$T$, find the smallest power $\alpha(p)$ such that
\[ \|Tf\|_{L^p(w)} \leq C_{n,T}[w]_{A_p}^{\alpha(p)}\|f\|_{L^p(w)}. \]
This problem was first investigated by Buckley~\cite{MR1124164}.
More recently, it has attracted renewed attention because of the
work of Astala, Iwaniec and Saksman~\cite{MR1815249}.  They proved
that sharp regularity results for solutions to the Beltrami equation
hold provided that  the Beurling-Ahlfors operator satisfies
$\alpha(p)=1$  for $p>2$.

The problem of characterizing the weights that govern the two-weight
norm inequalities for classical operators is still open and there
are several approaches to finding sufficient conditions on weights for
an operator to be bounded from $L^p(v)$ to $L^p(u)$.  One approach
is to replace the two-weight $A_p$ condition with the $A_p$ ``bump''
condition:
\[ \sup_Q \|u^{1/p}\|_{A,Q} \|v^{-1/p}\|_{B,Q} < \infty, \]
where $A$ and $B$ are Young functions and the norms are localized
Orlicz norms slightly larger than the $L^p$ and $L^{p'}$ norms.
(Precise definitions will be given below.)   Sufficient growth
conditions on $A$ and $B$ are known for many operators and this has
led to a number of conjectures on sharp sufficient conditions.   For
the history of this approach we refer the reader
to~\cite{MR2351373,cruz-martell-perezBook,MR1713140,MR1793688,MR1991004}.

\smallskip

In this paper we develop a unified approach to both of these
problems and the results we get are sharp.    We consider one and
two-weight norm inequalities for singular integrals, maximal singular integrals, the dyadic
paraproduct, the dyadic square function and the vector-valued
maximal operator.  The results in the one-weight case for singular
integrals are not new, but we believe that our proofs are simpler
than existing proofs.   The remaining theorems, however, are all
new.

We believe that our approach shows that there is a
deep connection between sharp results in the one and two-weight
case.  Further, key to our approach is that the operators are either dyadic or can be approximated by dyadic operators (e.g., by the Haar shift operators defined below).  Thus our results will extend to any operator that can be approximated in this way.

\subsection*{Singular integrals}
It is conjectured that if $T$ is any Calder\'on-Zyg\-mund singular
integral operator, then for any $p$, $1<p<\infty$, and for any $w\in
A_p$,
\begin{equation} \label{eqn:sio-onewt}
\|Tf\|_{L^p(w)} \leq C_{T,n,p}\,
[w]_{A_p}^{\max\left(1,\frac{1}{p-1}\right)}\|f\|_{L^p(w)}.
\end{equation}
This inequality is true if $T$ is  the Hilbert transform, a Riesz
transform or the Beurling-Ahlfors operator.

\begin{theorem} \label{thm:sio-onewt}
Given $p$, $1<p<\infty$, if $T$ is the  Hilbert transform, a Riesz
transform or  the Beurling-Ahlfors operator, then for all $w\in A_p$
inequality \eqref{eqn:sio-onewt} holds.
\end{theorem}

This result was first proved by
Petermichl~\cite{MR2354322,MR2367098} and Petermichl and
Volberg~\cite{MR1894362}.   For each operator the proof requires
several steps.  First, it is enough to prove  the case $p=2$; the
other values of $p$ follow from a version of the Rubio de Francia
extrapolation theorem with sharp constants due to
Dragi{\v{c}}evi{\'c} {\em et al.}~\cite{MR2140200}
(Theorem~\ref{thm:sharp-extrapol} below).    Second, for each of the
above operators the problem is reduced to proving the weighted $L^2$
inequality for a corresponding dyadic operator by proving that the
given operator can be approximated by integral averages of the
dyadic operators (and their analogs defined on translations and
dilations of the standard dyadic grid).  Finally, the desired
inequality was proved for each of these dyadic operators using
Bellman function techniques and two-weight norm inequalities.

Recently, Lacey, Petermichl and
Reguera-Rodriguez~\cite{lacey-petermichl-reguera2010}
gave a proof of the sharp $A_2$ constant for a large family of Haar
shift operators that includes all of the dyadic operators needed for
the above results.  Their proof avoids the use of Bellman functions,
and instead uses a deep, two-weight ``$Tb$ theorem'' for Haar shift
operators due to Nazarov, Treil and Volberg~\cite{MR2407233}.

We give a different and simpler proof that uses approximation
by dyadic Haar shifts but avoids both Bellman functions and two-weight
norm inequalities such as the $Tb$ theorem. Instead, we use a very
interesting decomposition argument based on local mean oscillation
recently developed by Lerner [30] to prove the corresponding result
for dyadic Haar shifts. Intuitively, this decomposition
may be thought of as a version of the Calder\'on-Zygmund
decomposition of a function, replacing the mean by the median. (We
will make this more precise below.)  Theorem 1.1 was
announced in \cite{MR2628851}.

\begin{remark}
After this paper was completed we learned of several
other related results.  First, Vagharsyakhan~\cite{vargharsyakhan}
has shown that in one dimension, all convolution-type
Calder\'on-Zygmund singular integral operators with sufficiently
smooth kernel can be approximated by Haar shifts.  Second, Lacey
{\em et al.}~\cite{lacey+5P} used a deep characterization of the
one-weight problem in~\cite{perez-treil-volbergP} to prove Theorem~\ref{thm:sio-onewt} for
all singular integrals with sufficiently smooth kernels.  Third,
Lerner~\cite{lernerP2010} proved Theorem~\ref{thm:sio-onewt} for
any convolution-type Calder\'on-Zygmund singular integrals
provided $p \geq 3$ or $1<p\leq 3/2$.  Finally,
Hyt\"onen~\cite{hytonenP2010} proved Theorem~\ref{thm:sio-onewt}
for all singular integrals and all $p>1$, thus solving the
so-called $A_2$ conjecture.  His proof is extremely technical: it
is based on the approach in~\cite{perez-treil-volbergP} and a
refinement of the arguments
in~\cite{lacey-petermichl-reguera2010}. A simpler proof of the
$A_2$ conjecture based upon the previous three papers appears
in~\cite{hytonen-perez-treil-volbergP}.
\end{remark}

An important advantage of our approach is that it also yields sharp
two-weight norm inequalities.   To state our result we need a few
definitions.  A Young function is a function $A : [0,\infty)
\rightarrow [0,\infty)$ that is continuous, convex and strictly
increasing, $A(0)=0$ and $A(t)/t\rightarrow \infty$ as $t\rightarrow
\infty$. Given a cube $Q$ we define the localized Luxemburg norm by
\[  \|f\|_{A,Q} = \inf \left\{ \lambda > 0 : \avgint_Q A\left(\frac{|f(x)|}{\lambda}\right)dx \leq 1 \right\}.  \]
When $A(t)=t^p$, $1<p<\infty$, we write
$$
\|f\|_{p,Q}=\left(\avgint_Q |f(x)|^p dx\right)^{1/p}.
$$
The associate function of $A$ is the Young function
\[ \bar{A}(t) = \sup_{s>0}\{ st - A(s)\}. \]
A Young function $A$ satisfies the $B_p$ condition if for some $c>0$,
\[ \int_c^\infty \frac{A(t)}{t^p} \frac{dt}{t} < \infty. \]
Important examples of such functions are of the form
$A(t)=t^{p}\log(e+t)^{-1-\epsilon}$, $\epsilon>0$, which have
associate functions $\bar{A}(t)\approx
t^{p'}\log(e+t)^{p'-1+\delta}$, $\delta>0$.

\begin{theorem} \label{thm:sio-twowt}
Given $p$, $1<p<\infty$, let $A$ and $B$ be Young functions such
that $\bar{A}\in B_{p'}$ and $\bar{B}\in B_p$.  Then for any pair of
weights $(u,v)$ such that
\begin{equation} \label{eqn:sio-twowt1}
\sup_Q \|u^{1/p}\|_{A,Q}\|v^{-1/p}\|_{B,Q} < \infty,
\end{equation}
we have that
\begin{equation} \label{eqn:sio-twowt2}
\| Tf \|_{L^p(u)} \leq C\|f\|_{L^p(v)},
\end{equation}
where $T$ is the Hilbert transform, a Riesz transform, or the
Beurling-Ahlfors operator.
\end{theorem}

Condition \eqref{eqn:sio-twowt1} is referred to as an
$A_p$ bump condition: when $A(t)=t^p$ and $B(t)=t^{p'}$, we get the
two-weight $A_p$ condition. Theorem~\ref{thm:sio-twowt} was proved
in~\cite{MR2351373} for the Hilbert transform in the special case
that $A(t)=t^p\log(e+t)^{p-1+\delta}$, $\delta>0$ (here $\bar{A}\in B_{p'}$), and for Riesz
transforms (indeed, for any Calder\'on-Zygmund singular integral)
given the additional hypothesis that $p>n$.  Examples
(see~\cite{cruz-martell-perezBook,MR1713140}) show that in this
particular case
these results are sharp, since they are false in general if we take
$\delta=0$ (when $\bar{A}\not\in B_{p'}$).  Theorem~\ref{thm:sio-twowt} was proved for the Hilbert
transform and general singular integrals when $p>n$ by
Lerner~\cite{lernerP2009} by combining his decomposition argument
with the arguments in~\cite{MR2351373}.

Two-weight inequalities were
first considered by Muckenhoupt~\cite{muckenhoupt72}, who noted that the same
proof as in the one-weight case immediately shows that for all
$p$, $1\leq p <\infty$, $(u,v)\in A_p$ if and only if the maximal
operator satisfies the weak $(p,p)$ inequality.  However, Muckenhoupt and
Wheeden~\cite{muckenhoupt-wheeden76}  soon showed that
while the two-weight $A_p$ condition is necessary for the strong
$(p,p)$ inequality for the maximal operator and the strong and
weak type inequalities for the Hilbert transform, it is not
sufficient. This led Muckenhoupt and
Wheeden to focus not on the structural or geometric properties of
$A_p$ weights but on their relationship to the maximal operator,
in particular, the fact that $w\in A_p$ was necessary and
sufficient for the maximal operator to be bounded on $L^p(w)$ and
$L^{p'}(w^{1-p'})$.  They made the following
conjecture that is still open: a
sufficient condition for the Hilbert transform to satisfy the
strong $(p,p)$ inequality $H : L^p(v) \rightarrow L^p(u)$,
$1<p<\infty$, is that the maximal operator satisfies the pair of
inequalities
\begin{equation}
M : L^p(v) \rightarrow L^p(u),
\qquad
M : L^{p'}(u^{1-p'}) \rightarrow L^{p'}(v^{1-p'}). \label{eqn:MW}
\end{equation}

Bump $A_p$ conditions were first considered by
Neugebauer~\cite{neugebauer83} who showed the following striking
result: a pair of weights $(u,v)$ satisfies
\eqref{eqn:sio-twowt1} with power bumps $A(t)=t^{r\,p}$,
$A(t)=t^{r\,p'}$ for some $r>1$ if and only if there exist $w\in
A_p$ and positive constants $c_1,\,c_2$ such that $c_1u(x) \leq w(x)
\leq c_2v(x)$. From this condition we immediately get a large number
of two-weight norm inequalities as corollaries to the analogous
one-weight results.  In particular, we get the two inequalities
\eqref{eqn:MW}. An immediate question was whether this condition
could be weakened and still get that the maximal operator satisfies
$M : L^p(v)\rightarrow L^p(u)$.  This was answered
in~\cite{perez95}, where it was shown that a sufficient condition
for \eqref{eqn:MW} was that the pair of weights satisfies
\eqref{eqn:sio-twowt1} with $\bar{A}\in B_{p'}$ and $\bar{B}\in
B_p$. The centrality of these $B_p$ conditions is shown by the fact
that they are sharp within the scale of Orlicz bumps as shown in
\cite{perez95}.  This led naturally to the following version of the
conjecture of Muckenhoupt and Wheeden: a sufficient condition on the
pair of weights $(u,v)$ for any singular integral to satisfy $T :
L^p(v)\rightarrow L^p(u)$ is that \eqref{eqn:sio-twowt1}
holds.  Progress on this conjecture was made in
\cite{MR1991004, MR2351373, lernerP2009}. Theorem
\ref{thm:sio-twowt} completely solves it for the
Hilbert and Riesz transforms and the Beurling-Ahlfors operator, and
as we noted above it is the best possible result  in
the scale of $B_p$ bumps. See \cite{cruz-martell-perezBook}
for further details and
references on this topic.

In the past decade, a great deal of attention has been focused on
proving that ``testing conditions'' are necessary and sufficient for
two-weight norm inequalities for singular integrals.  (See Nazarov,
Treil and Volberg~\cite{nazarov-treil-volberg99,volberg03,MR2407233}
and the recent preprints by Lacey, Sawyer and
Uriarte-Tuero~\cite{lacey-sawyer-uriarteP,lacey-sawyer-uriarteP2}.)  More precisely,
given a singular integral $T$, it is conjectured that $T:
L^p(v)\rightarrow L^p(u)$ if and only if for every cube $Q$,
\begin{gather*}
\int_Q |T(v^{1-p'} \chi_Q)(x)|^p u(x)\,dx \leq C\int_Q v(x)^{1-p'}\,dx
\\
\int_Q |T(u\chi_Q)(x)|^{p'} v(x)^{1-p'}\,dx \leq C\int_Q u(x)\,dx.
\end{gather*}
The necessity of these conditions is immediate.  The best known
results are for $p=2$; partial results (with additional hypotheses)
are known for other values of $p$.   These results are of great
interest not only because of the elegance of this conjecture but also
because of their connection with $Tb$-theorems on non-homogeneous spaces
(see~\cite{volberg03} and the references it contains).

Testing conditions and $A_p$ bump conditions are not
readily comparable: they represent two fundamentally different
approaches to the two-weight problem.  While both approaches are
important, we believe that bump conditions have several advantages
over testing conditions.  First, they are universal, geometric
conditions: they are independent of the operators and any pair
yields norm inequalities for a range of operators.  Second, they are
much easier to check than the testing conditions, and it is very
easy to construct examples of weights that do and do not satisfy a
given bump condition.  (For many examples and a general technique
for constructing them, see~\cite{cruz-martell-perezBook}.)  Third,
they are not tied to $L^2$, unlike testing conditions where the
transition from $p=2$ to all $p$ has proved to be very difficult.
(In this regard, we note that in~\cite{nazarov-treil-volberg99} it
was claimed---without proof---that in the specific case they were
considering, testing conditions were not sufficient.)

\subsection*{Maximal singular integrals}
Given a singular integral $T$ with convolution kernel $K$, recall
that the associated maximal singular integral is defined by
\[ T_*f(x) = \sup_{\epsilon> 0} |T_\epsilon f(x)| =
\sup_{\epsilon> 0} \left|\int_{|y|>\epsilon}
K(y)f(x-y)\,dy\right|. \]
Somewhat surprisingly, both Theorem~\ref{thm:sio-onewt} and
Theorem~\ref{thm:sio-twowt} remain true if the singular integral is
replaced by the associated maximal singular integral.

\begin{theorem} \label{thm:max-singular}
Given $p$, $1<p<\infty$, and $w\in A_p$, then
inequality~\eqref{eqn:sio-onewt} holds if $T$ is replaced by $T_*$,
where $T$ is the Hilbert transform, a Riesz transform or the
Beurling-Ahlfors operator.    Similarly, if the pair $(u,v)$
satisfies \eqref{eqn:sio-twowt1}, then inequality
\eqref{eqn:sio-twowt2} holds if $T$ is replaced by $T_*$.
\end{theorem}

In the one-weight case, Theorem~\ref{thm:max-singular} was proved
very recently by Hyt\"onen {\em et al.}
\cite{hytonen-lacey-reguera-vagharshakyanP2009}.   Their proof used
a very general family of ``maximal'' dyadic shift operators and a
characterization of the two-weight norm inequalities for maximal
singular integrals due to Lacey, Sawyer and
Uriarte-Tuero~\cite{lacey-sawyer-uriarteP}.  In the two-weight case
this result is new.   In both the one and two-weight case our
approach is to prove the corresponding result for the associated
maximal dyadic shift operators.

\begin{remark}
Very recently, Lerner~\cite{lernerP2010} has proved Theorem~\ref{thm:max-singular} in the one-weight case for
general Calder\'on-Zygmund maximal singular integral operators when $p>3$.
\end{remark}

\subsection*{Dyadic paraproducts and constant Haar multipliers}
Let $\Delta$ denote the collection of dyadic cubes in $\re$. We
consider two operators defined on the real line. A function $b$ is
in dyadic $BMO$, we write $b\in BMO^d$, if
\[ \|b\|_{*,d} = \sup_{I\in\Delta} \left(\avgint_I |b(x)-b_I|^2\,dx\right)^{1/2} <
\infty, \]
where $b_I =\avgint_I b(x)\,dx$. Given a dyadic interval $I$, $I_+$
and $I_-$ are its right and left halves, and the Haar function $h_I$
is defined by
\[ h_I(x)  =  |I|^{-1/2}\big(\chi_{I_-}(x) - \chi_{I_+}(x)\big). \]
Define the dyadic paraproduct $\pi_b$ by
\[ \pi_b f(x) = \sum_{I\in \Delta} f_I \la b, h_I \ra h_I(x).  \]
For an overview of the history and properties of the dyadic
paraproduct, we refer the reader to Pereyra~\cite{MR1864538}.

\begin{theorem} \label{thm:paraproduct}
Given a function $b\in BMO^d$, and $p$, $1<p<\infty$, then for all
$w\in A_p$,
\[ \|\pi_b f\|_{L^p(w)} \leq
C_{p} \|b\|_{*,d}\,
[w]_{A_p}^{\max\left(1,\frac{1}{p-1}\right)}\|f\|_{L^p(w)}. \]
Furthermore, given a pair $(u,v)$ that satisfies
\eqref{eqn:sio-twowt1}, then
\[ \|\pi_b f\|_{L^p(u)} \leq
C\|b\|_{*,d}\|f\|_{L^p(v)}. \]
\end{theorem}

In the one-weight case, Theorem~\ref{thm:paraproduct} was first proved by
Beznosova~\cite{MR2433959} using Bellman function techniques.   A
different proof that avoided Bellman functions but used two-weight
inequalities was given in~\cite{hytonen-lacey-reguera-vagharshakyanP2009}.

\bigskip

Given a sequence $\alpha=\{\alpha_I\}_{I\in\Delta}\in\ell^\infty$,
define the constant Haar multiplier  $T_\alpha$ by
$$
T_\alpha f(x) =\sum_{I\in\Delta} \alpha_{I}\,\la
f,h_I\rangle\,h_I(x).
$$
If $\alpha_I=1$, then $T_\alpha$ is the identity operator.   For
more on the properties of these operators, see
Pereyra~\cite{MR1864538}. The analog of
Theorem~\ref{thm:paraproduct} is true for constant Haar multipliers.

\begin{theorem} \label{thm:haarmult}
Given a sequence $\alpha = \{\alpha_I\}_{I\in\Delta} \in
\ell^\infty$, and $p$, $1<p<\infty$, then for all $w\in A_p$,
\[ \|T_\alpha f\|_{L^p(w)} \leq
C_{p} \|\alpha\|_{\ell^\infty}\,
[w]_{A_p}^{\max\left(1,\frac{1}{p-1}\right)}\|f\|_{L^p(w)}. \]
Furthermore, given a pair $(u,v)$ that satisfies
\eqref{eqn:sio-twowt1}, then
\[ \|T_\alpha f\|_{L^p(u)} \leq
C \|\alpha\|_{\ell^\infty}\|f\|_{L^p(v)}. \]
\end{theorem}

In the special case when $\alpha_I=\pm 1$, Theorem~\ref{thm:haarmult}
was proved by Wittwer~\cite{MR1748283}.

\subsection*{Dyadic square functions}
Let $\Delta$ denote the collection of dyadic cubes in $\re^n$.
Given $Q\in \Delta$, let $\widehat{Q}$ be its dyadic parent:  the
unique dyadic cube containing $Q$ whose side-length is twice that of
$Q$.  The dyadic square function is the operator
\[ S_df(x) = \left(\sum_{Q\in\Delta}(f_Q-f_{\widehat{Q}})^2\chi_Q(x)\right)^{1/2}, \]
where
$ f_Q = \avgint_Q f(x)\,dx$.
For the properties of the dyadic square function we refer the reader
to Wilson~\cite{wilson07}.

\begin{theorem} \label{thm:square}
Given $p$, $1<p<\infty$, then for any $w\in A_p$,
\[ \|S_d f\|_{L^p(w)} \leq
C_{n,p}[w]_{A_p}^{\max\left(\frac12,\frac{1}{p-1}\right)}\|f\|_{L^p(w)}.
\]
Further, the exponent $\max\left(\frac{1}{2},\frac{1}{p-1}\right)$
is the best possible.
\end{theorem}

The exponent in Theorem \ref{thm:square} was first conjectured
by Lerner~\cite{MR2200743} for the continuous square function; he
also showed it was the best possible.   In \cite{MR2524658} he proved that
for $p>2$ the sharp exponent is at most $p'/2>\max(\frac{1}{2},\frac{1}{p-1})$.  When $p=2$,
Theorem~\ref{thm:square} was proved by Wittwer~\cite{MR1897458} and by
Hukovic, Treil and Volberg~\cite{MR1771755};
this was extended to $p<2$ by extrapolation in \cite{MR2140200} and
examples were given to show that in this range the exponent is the
best possible.

\begin{remark}
Very recently, Lerner~\cite{lernerP2010} proved the analog of
Theorem~\ref{thm:square} for continuous square functions.  His proof
uses the intrinsic square function introduced by
Wilson~\cite{wilson07}.
\end{remark}

\begin{theorem} \label{thm:Bp-bump:Sd}
Fix $p$, $1<p<\infty$.  Suppose $1<p \leq 2$, and  $B$ is a Young
function such that $\bar{B}\in B_p$,  If the pair $(u,v)$ satisfies
\begin{equation} \label{eqn:Sd1}
  \sup_Q \|u^{1/p}\|_{p,Q} \|v^{-1/p}\|_{B,Q}  < \infty,
\end{equation}
then
\begin{equation} \label{eqn:Sd2}
 \|S_d f\|_{L^p(u)} \leq C\|f\|_{L^p(v)}.
\end{equation}
Suppose $2<p<\infty$, and $A$ and $B$  are Young functions such that
$\bar{A}\in B_{(p/2)'}$ and $\bar{B}\in B_p$.   If the pair $(u,v)$
satisfies
\begin{equation}\label{eqn:bp-bumpSqrt}
\sup_Q \|u^{2/p}\|_{A,Q}^{1/2} \|v^{-1/p}\|_{B,Q} < \infty,
\end{equation}
then \eqref{eqn:Sd2} holds.
\end{theorem}

Condition \eqref{eqn:Sd1} is the same condition for the maximal
operator to map $L^p(v)$ to $L^p(u)$, whereas condition
\eqref{eqn:bp-bumpSqrt} is more similar to the conditions needed for
singular integrals, but with a smaller ``bump'' on the left.  This
is easier to see in the scale of ``log bumps.''  As we noted above,
for a singular integral we need to take
$A(t)=t^p\log(e+t)^{p-1+\delta}$, $\delta>0$, but for the dyadic square function
it suffices to take $A(t)=t^{p/2}\log(e+t)^{p/2-1+\delta}$,
which after rescaling leads to $t^{p}\log(e+t)^{p/2-1+\delta}$.    This
difference in the behavior of the dyadic square function depending
on whether $p\leq 2$ or $p>2$ was first noted in
\cite{cruz-martell-perezBook}.  There we conjectured
Theorem~\ref{thm:Bp-bump:Sd} was true and proved it in some special
cases.  Furthermore, we proved   in~\cite{cruz-martell-perezBook} that this result is sharp in the scale
of log bumps:  if we take $p>2$ and let $A(t)=t^{p/2}\log(e+t)^{p/2-1+\delta}$, then the theorem holds for
$\delta>0$, when $A\in B_{(p/2)'}$, but not for $\delta=0$, when $A\not\in B_{(p/2)'}$.

\begin{remark}
Theorem~\ref{thm:square} remains true if we replace the $A_p$
condition by the dyadic $A_p$ condition (i.e., defined only with
respect to dyadic cubes).    Similarly, Theorem~\ref{thm:Bp-bump:Sd}
remains true if the weight conditions are restricted to dyadic cubes.
For both theorems this follows by examining the proofs and details are
left to the interested reader.
\end{remark}

\subsection*{The vector-valued maximal operator}
Let $M$ be the Hardy-Little\-wood maximal operator.  Given a
vector-valued function $f=\{ f_i\}$, and $q$, $1<q<\infty$, define
the vector-valued maximal operator $\M_q$ by
\[ \M_qf(x) = \left(\sum_{i=1}^\infty Mf_i(x)^q\right)^{1/q}. \]
The vector-valued maximal operator was introduced by C.
Fefferman and Stein~\cite{fefferman-stein71}; for more information
see \cite{MR807149}.

Similar to the dyadic square function, the behavior of the
vector-valued maximal operator depends on the relative sizes of $p$
and $q$.

\begin{theorem} \label{thm:vvmax-onewt}
Fix $q$, $1<q<\infty$.  Given $p$, $1<p<\infty$, if $w\in A_p$, then
\[ \|\M_q f \|_{L^p(w)} \leq C_{p,q,n}[w]_{A_p}^{\max(\frac{1}{q},\frac{1}{p-1})}
\left(\int_{\re^n} \|f(x)\|_{\ell^q}^pw(x)\,dx\right)^{1/p}. \]
Further, the exponent $\max\left(\frac{1}{q},\frac{1}{p-1}\right)$
is the best possible.
\end{theorem}

Theorem~\ref{thm:vvmax-onewt} is new.    A slightly worse bound than
Theorem~\ref{thm:vvmax-onewt} was implicit in the literature, but does
not appear to have been stated explicitly.    To get this weaker
estimate, note that when $q=p$, by the sharp result for the
Hardy-Littlewood maximal operator we have that
\[ \|\M_q f \|_{L^q(w)} \leq C_{q,n}[w]_{A_q}^{\frac{1}{q-1}}
\left(\int_{\re^n} \|f(x)\|_{\ell^q}^qw(x)\,dx\right)^{1/q}. \]
Then by adapting to this context the sharp
version of the Rubio de Francia extrapolation theorem, we get the
exponent $\frac{1}{q-1}$ if $p>q$ and $\frac{1}{p-1}$ if
$p<q$. Theorem~\ref{thm:vvmax-onewt} improves this bound for $p$
large.


\begin{theorem} \label{thm:vvmax-twowt}
Fix $q$, $1<q<\infty$.  Suppose $1<p \leq q$, and  $B$ is a Young
function such that $\bar{B}\in B_p$.  If the pair $(u,v)$ satisfies
\begin{equation} \label{eqn:vvmax1}
  \sup_Q \|u^{1/p}\|_{p,Q} \|v^{-1/p}\|_{B,Q}  < \infty,
\end{equation}
then
\begin{equation} \label{eqn:vvmax2}
 \|\overline{M}_q f\|_{L^p(u)} \leq C
\left(\int_{\re^n} \|f(x)\|_{\ell^q}^pv(x)\,dx\right)^{1/p}.
\end{equation}
Suppose $q<p<\infty$, and $A$ and $B$  are Young functions such that
$\bar{A}\in B_{(p/q)'}$ and $\bar{B}\in B_p$.   If the pair $(u,v)$
satisfies
\begin{equation}\label{eqn:vvmax3}
\sup_Q \|u^{q/p}\|_{A,Q}^{1/q} \|v^{-1/p}\|_{B,Q} < \infty,
\end{equation}
then \eqref{eqn:vvmax2} holds.
\end{theorem}

When $p>q$, Theorem~\ref{thm:vvmax-twowt} is sharp in the scale of
log bumps.  The example in~\cite{cruz-uribe-perez00} shows that if
$A(t)=t^{p/q}\log(e+t)^{p/q-1+\delta}$, then the result fails if $\delta=0$, when
$A\not\in B_{(p/q)'}$. (If    $\delta>0$, $A\in B_{(p/q)'}$). 
In the case $p\leq q$ Theorem~\ref{thm:vvmax-twowt} is not new: it
was first proved in~\cite{perez00}; for a different proof
see~\cite{cruz-martell-perezBook}.   We include it here for
completeness and to highlight the similarity to the dyadic square
function.  The case $p>q$ is new; it was first conjectured
in~\cite{cruz-martell-perezBook} where a few special cases were
proved.

\begin{remark}
To obtain Theorems~\ref{thm:vvmax-onewt} and \ref{thm:vvmax-twowt} we first consider the corresponding dyadic vector-valued maximal operator and establish both results for it. In such a case, as observed before for the dyadic square function, we ca replace the $A_p$
condition by the dyadic $A_p$ condition, and in \eqref{eqn:vvmax1}, \eqref{eqn:vvmax3} the sup can be taken over all dyadic cubes. Further details are
left to the interested reader.
\end{remark}

\medskip

\subsection*{Organization}
The remainder of this paper is organized as follows.  In
Section~\ref{section:prelim} we gather some basic results,
primarily about weighted norm inequalities, that are needed in
subsequent sections.   In Section~\ref{section:lerner} we give some
preliminary material about the local mean oscillation of a function
and state the decomposition theorem of Lerner.   In
Section~\ref{section:dyadic} we define the Haar shift operators and
prove the key estimate we need to apply Lerner's results.  In
Sections~\ref{section:main-proof}--\ref{section:Mq} we prove our
main results.   In Section~\ref{section:main-proof} we prove our
results for singular integrals by proving the corresponding results
for Haar shift operators.  As a corollary to these theorems we get
our results for dyadic paraproducts and constant Haar multipliers.
After the proof of Theorems~\ref{thm:sio-onewt}
and~\ref{thm:sio-twowt} we will briefly discuss the technical
obstructions which prevent us from applying  our approach directly
to a Calder\'on-Zygmund singular integral. The proof for maximal
singular integrals is very similar to the proof for singular
integrals, but since we introduce a new family of dyadic operators
we give these results in Section~\ref{section:maximal}.  The proofs
for square functions and vector-valued maximal operators are also
very similar to those for singular integrals, so we will only sketch
the proofs of these results in Sections~\ref{section:Sd} and
\ref{section:Mq}, highlighting the key changes.

\medskip

We would like to thank Andrei Lerner and
Michael Wilson for a clarifying discussion about the results in
Section~\ref{section:lerner}.

\section{Preliminary results}
\label{section:prelim}

In this section we state some basic results that we will need in our
proofs. The first is the sharp one-weight bound for the
Hardy-Littlewood maximal operator.  This result is due
to Buckley~\cite{MR1124164}; for an elementary proof, see
Lerner~\cite{MR2399047}.

\begin{theorem} \label{thm:sharpmax}
Given $p$, $1<p<\infty$, and any $w\in A_p$,
\[ \|Mf\|_{L^p(w)} \leq C_n\,{(p')^{1/p}(p)^{1/p'}}\,  [w]_{A_p}^{\frac{1}{p-1}}\,\|f\|_{L^p(w)}. \]
\end{theorem}

The next result is the sharp version of the Rubio de Francia
extrapolation theorem due to  Dragi{\v{c}}evi{\'c} {\em et
al.}~\cite{MR2140200}.

\begin{theorem} \label{thm:sharp-extrapol}
Suppose that for some $p_0$, $1<p_0<\infty$, there exists
$\alpha(p_0)>0$ such that for every $w\in A_{p_0}$, a sublinear
operator $T$ satisfies
\[ \|Tf\|_{L^{p_0}(w)} \leq C_{n,T,p_0}[w]_{A_{p_0}}^{\alpha(p_0)}\|f\|_{L^{p_0}(w)}. \]
Then for every $p$, $1<p<\infty$,
\[ \|Tf\|_{L^{p}(w)} \leq C_{n,T,p_0,p}[w]_{A_{p}}^{\alpha(p_0)\,\max\left(1,\frac{p_0-1}{p-1}\right)}\|f\|_{L^{p}(w)}. \]
\end{theorem}

Third, we need a norm inequality for a weighted dyadic maximal
operator.

\begin{lemma} \label{lemma:wtdmax}
Let $\sigma$ be a locally integrable function such that
$\sigma>0$ a.e., and define the weighted dyadic maximal operator
\[ M_\sigma^df(x) = \sup_{\substack{Q\in \Delta \\ x\in Q}}
\frac{1}{\sigma(Q)} \int_Q |f(y)|\sigma(y)\,dy. \]
Then for all $p$, $1<p<\infty$,
\[ \|M_\sigma^d f\|_{L^p(\sigma)} \leq p'\|f\|_{L^p(\sigma)}. \]
In particular, the constant is independent of $\sigma$.
\end{lemma}

This result is well known and follows by standard arguments.
Clearly,  $M_\sigma^d$ is bounded on $L^\infty(\sigma)$ with
constant $1$.   It is also of weak-type $(1,1)$ (with respect to
$\sigma$) with constant $1$; this follows from the dyadic
structure. Then by Marcinkiewicz interpolation  we get the desired
estimate (see \cite[Chapter 1, Exercise 1.3.3]{MR2445437}).

In the two-weight case we will need a norm inequality for Orlicz
maximal operators.  This result was proved in~\cite{perez95}.

\begin{theorem} \label{thm:orlicz-max}
Given $p$, $1<p<\infty$, suppose that $A$ is a Young function such
that $A\in B_p$.   Then the Orlicz maximal operator
\[ M_A f(x) = \sup_{x\ni Q} \|f\|_{A,Q} \]
is bounded on $L^p(\re^n)$.
\end{theorem}

We will also need a two-weight norm inequality for the maximal
operator, also from~\cite{perez95}.

\begin{theorem} \label{thm:twowt-max}
Given $p$, $1<p<\infty$, suppose that $B$ is a Young function such
that $\bar{B}\in B_p$.  Then for any pair $(u,v)$ that satisfies
\[ \sup_Q \|u^{1/p}\|_{p,Q}\|v^{-1/p}\|_{B,Q} < \infty, \]
we have that
\[ \|Mf \|_{L^p(u)} \leq C\|f\|_{L^p(v)}. \]
\end{theorem}

To apply Theorem~\ref{thm:twowt-max} we will need to use two facts
about Young functions and Orlicz norms.  First,  if $A$ is a Young
function such that $\bar{A} \in B_{p'}$, then $\bar{A}(t) \le C\,
t^{p'}$ for $t\ge 1$, so that $t^p\le A(C\,t)$ for $t\ge 1$, and
therefore, $\|u^{1/p}\|_{p,Q}\leq C\|u^{1/p}\|_{A,Q}$.   Second,
given a Young function $A$ we have the generalized H\"older's
inequality,
\[ \avgint_Q |f(x)g(x)|\,dx \leq 2 \|f\|_{A,Q}\|g\|_{\bar{A},Q}. \]
See~\cite{cruz-martell-perezBook} for more details.

\section{Local mean oscillation}
\label{section:lerner}

To state Lerner's decomposition argument we must first make some
definitions and give a few basic results.    We follow the
terminology and notation in~\cite{lernerP2009}, which in turn is
based on Fujii~\cite{MR946637,MR1115188} and Jawerth and
Torchinsky~\cite{MR779906}.  We note in passing that many of the
underlying ideas originated in the work of Carleson~\cite{MR0477058}
and Garnett and Jones~\cite{MR658065}.

Hereafter we assume that all functions $f$ are measurable and
finite-valued almost everywhere.  Given a cube $Q$ and $\lambda$,
$0<\lambda<1$, define the local mean oscillation of $f$ on $Q$ by
\[ \omega_\lambda(f,Q) = \inf_{c\in \re }
\big((f-c)\chi_Q\big)^*(\lambda |Q|), \]
where $f^*$ is the (left-continuous) non-increasing rearrangement of $f$:
\[ f^*(t) = \inf\big\{ \alpha> 0 : |\{ x \in \re^n : |f(x)|>\alpha\}| <
t \big\}. \]
The local
sharp maximal function of $f$ relative to $Q$ is then defined by
\[ M^\#_{\lambda,Q} f(x) = \sup_{\substack{Q'\ni x\\Q'\subset Q}}
\omega_\lambda(f,Q).  \]
The local sharp maximal function is significantly smaller than the
C.~Fefferman-Stein sharp maximal function: for all
$\lambda>0$ sufficiently small, $M(M^\#_{\lambda,Q}f)(x) \leq
C(n,\lambda)M^\#f(x)$. (See~\cite{MR779906}.)

A median value of $f$ on $Q$ is a (possibly not unique) number
$m_f(Q)$ such that
\[
\max\big(| \{ x\in Q : f(x) > m_f(Q) \}|,  | \{ x\in Q : f(x) <
m_f(Q) \}|\big) \leq \frac{|Q|}{2}.
\]
(A different but functionally equivalent definition is given
in~\cite{MR1115188}.)  The median plays the same role for the local
sharp maximal function as the mean does for the C.~Fefferman-Stein
sharp maximal function.  More precisely, for each $\lambda$,
$0<\lambda \leq 1/2$,
\[ \omega_\lambda(f,Q) \leq  \big((f-m_f(Q))\chi_Q\big)^*(\lambda
|Q|) \leq 2\omega_\lambda(f,Q).\]
(The first inequality is immediate; the second follows
from~\eqref{eqn:median-f*} below and the fact that for any constant
$c$, $m_f(Q)-c=m_{f-c}(Q)$;  see Lerner~\cite{lerner04}.)

To estimate the median and the local mean oscillation we need several
properties.  For the convenience of the reader we gather them as a
lemma and sketch their proofs.

\begin{lemma}
Given a measurable function $f$ and a cube $Q$, then
for all $\lambda$,  $0<\lambda<1$,
and $p$, $0<p<\infty$,
\begin{gather}
(f\chi_{Q })^*(\lambda |Q|)\le \lambda^{-1/p}\,
\|f\|_{L^{p,\infty}(Q,|Q|^{-1}dx)}, \label{eqn:mean-est1} \\
(f\chi_{Q })^*(\lambda |Q|)\le
\left(\frac{1}{\lambda|Q|}\int_Q|f|^p\,dx\right)^{1/p}.
\label{eqn:mean-est2}
\end{gather}
Furthermore,
\begin{equation}
|m_f(Q)| \leq (f\chi_Q)^*(|Q|/2); \label{eqn:median-f*}
\end{equation}
in particular, if $f\in L^p$ for any $p>0$, then
$m_f(Q)\rightarrow 0$ as $|Q|\rightarrow \infty$.
\end{lemma}

\begin{proof}
Inequality~\eqref{eqn:mean-est2} follows immediately
from~\eqref{eqn:mean-est1}.  To prove this inequality, fix
$\alpha<(f\chi_Q)^*(\lambda|Q|)$.  Then
\[ \lambda^{-1/p}\,\|f\|_{L^{p,\infty}(Q,|Q|^{-1}dx)} \geq
\lambda^{-1/p}|Q|^{-1/p} \alpha |\{ x\in Q : |f(x)|>\alpha \}|^{1/p}
\geq \alpha. \]
Since this is true for all such $\alpha$, \eqref{eqn:mean-est1}
follows at once.

To prove \eqref{eqn:median-f*} we consider two cases.  Suppose
$m_f(Q)\geq 0$.  Define
\[ m_+ = \sup\{ \beta \geq m_f(Q) : |\{x\in Q : f(x)< \beta \}| \leq
|Q|/2\}; \]
then $0 \leq m_f(Q)\leq m_+$ so it will be enough to prove $m_+\leq
f^*(|Q|/2)$.  Take any $\alpha>0$ such that
\[ |\{ x\in Q : |f(x)|>\alpha \}| < |Q|/2. \]
Then
\[ |\{ x\in Q : f(x) \leq \alpha \}| \geq |\{ x \in  Q : |f(x)|\leq
\alpha \}| > |Q|/2. \]
Hence, for any $\beta>\alpha$,
\[ |\{ x\in Q : f(x) < \beta \}| \geq  |\{ x\in Q : f(x) \leq \alpha
\}| > |Q|/2, \]
and so $\beta \geq m_+$.   Since this is true for all $\beta>\alpha$,
we have that $\alpha\geq m_+$, and taking the infimum of all such
$\alpha$ we get that $(f\chi_Q)^*(|Q|/2)\geq m_+$.

Finally if $m_f(Q)<0$, define $g(x)=-f(x)$.  Then $-m_f(Q)$ is a median
of $g$ and the previous case yields $|m_f(Q)|=-m_f(Q) \leq g^*(|Q|/2) = f^*(|Q|/2)$.
\end{proof}

\begin{remark}
Inequality \eqref{eqn:mean-est1} is central to our proofs as  it
allows us to use  weak $(1,1)$ inequalities directly in our
estimates.  By way of comparison, in~\cite{MR2351373} a key
technical difficulty resulted from having to use Kolmogorov's
inequality rather than the weak $(1,1)$ inequality for a singular
integral. Overcoming this is the reason the results there were
limited to log bumps.
\end{remark}

\bigskip

To state Lerner's decomposition theorem, we generalize our notation
slightly:  given a cube $Q_0$, let $\Delta(Q_0)$ be the collection
of dyadic cubes relative to $Q_0$. Given $Q\in \Delta(Q_0)$,  $Q\neq
Q_0$, let $\widehat{Q}$ be its dyadic parent:  the unique cube in
$\Delta(Q_0)$ containing $Q$ whose side-length is twice that of $Q$.

\begin{theorem}  \label{thm:lerner}
\textup{(\cite{lernerP2009})}\; Given a measurable function $f$ and
a cube $Q_0$, for each $k\geq 1$ there exists a (possibly empty)
collection of pairwise disjoint cubes $\{Q_j^k\} \subset
\Delta(Q_0)$ such that if $\Omega_k = \bigcup_j Q_j^k$, then
$\Omega_{k+1}\subset \Omega_k$ and $|\Omega_{k+1}\cap Q_j^k|\leq
\frac{1}{2}|Q_j^k|$.   Furthermore, for almost every $x\in Q_0$,
\[ |f(x)-m_f(Q_0)| \leq 4M^\#_{\frac{1}{4},Q_0}f(x) +
4\sum_{k,j}
\omega_{\frac{1}{2^{n+2}}}(f,\widehat{Q}_j^k)\chi_{Q_j^k}(x). \]
\end{theorem}

\begin{remark} \label{rem:disjoint}
If for all $j$ and $k$ we define $E_j^k=Q_j^k\setminus
\Omega_{k+1}$, then the sets $E_j^k$ are pairwise disjoint and
$|E_j^k|\geq \frac{1}{2}|Q_j^k|$.
\end{remark}

\begin{remark}
 Though it is not explicit in \cite{lernerP2009}, it follows at once
 from the proof that  we can
 replace $M^\#_{\frac{1}{4},Q_0}$ by the corresponding dyadic
 operator $M^{\#,d}_{\frac{1}{4},Q_0}$, where
\[ M^{\#,d}_{\lambda,Q}f(x) = \sup_{x\in Q'\in \Delta(Q)} \omega_\lambda (f,Q'). \]
\end{remark}

Intuitively, one may think of the cubes $\{Q_j^k\}$ as being the
analog of the Calder\'on-Zygmund cubes for the function $f-m_f(Q_0)$
but defined with respect to the median instead of the mean.  The
cubes $Q_j^k$ are maximal dyadic cubes with respect to a dyadic
local sharp maximal operator.   The terms on the right-hand side of
the above inequality then play a role like that of the good and bad
parts of the Calder\'on-Zygmund decomposition.  A key difference, of
course, is that while the Calder\'on-Zygmund decomposition is done at one ``scale,''
the above theorem requires that we estimate the local mean
oscillation of $f$ at all scales.

\section{The Haar shift operators}
\label{section:dyadic}

To prove Theorems~\ref{thm:sio-onewt} and~\ref{thm:sio-twowt} we
need to prove the corresponding inequalities for certain dyadic
operators that can be used to approximate the Hilbert transform, the
Riesz transforms and the Beurling-Ahlfors operator.   We follow the
approach used in~\cite{lacey-petermichl-reguera2010} and
consider simultaneously a family of dyadic operators---the Haar
shift operators---that contains all the operators we are interested
in.

Let $\Delta$ be the set of dyadic cubes in $\re^n$.  For our arguments
we properly need to consider the sets $\Delta_{s,t}$, $s\in \re^n$,
$t>0$, of translations and dilations of dyadic cubes.  However, it will be immediate
that all of our arguments for dyadic cubes extend to these more
general families, so without loss of generality we will restrict
ourselves to dyadic cubes.

We define a Haar function on a cube $Q\in \Delta$ to be a function
$h_Q$ such that
\begin{enumerate}
\item $\supp(h_Q) \subset Q$;
\item if $Q'\in \Delta$ and $Q'\subsetneq Q$, then $h_Q$ is constant on $Q'$;
\item $\|h_Q\|_\infty \leq |Q|^{-1/2}$;
\item $\int_Q h_Q(x)\,dx = 0$.
\end{enumerate}

Given an integer $\tau\ge 0$, a Haar shift operator of index $\tau$
is an operator of the form
\[ H_\tau f(x) = \sum_{Q\in \Delta} \sum_{\substack{Q',Q''\in \Delta(Q)\\ 2^{-\tau n}|Q|\leq |Q'|,|Q''|}} a_{Q',Q''} \la f, h_{Q'} \rangle h_{Q''}(x), \]
where $a_{Q',Q''}$ is a constant such that
\[ |a_{Q',Q''}| \leq C\,\left( \frac{|Q'|}{|Q|}\frac{|Q''|}{|Q|}\right)^{1/2}. \]
We say that $H_\tau$ is a CZ Haar shift operator if it is bounded on
$L^2$.

An important example of a Haar shift operator when $n=1$ is the Haar
shift (also known as the dyadic Hilbert transform)  $H^d$, defined
by
\[ H^df(x) = \sum_{I\in \Delta} \langle f, h_I\rangle \big(
h_{I_-}(x) - h_{I_+}(x)\big), \]
where, as before, given a dyadic interval $I$, $I_+$ and $I_-$ are
its right and left halves, and
\[ h_I(x)  =  |I|^{-1/2}\big(\chi_{I_-}(x) - \chi_{I_+}(x)\big). \]
Clearly $h_I$ is a Haar function on $I$ and one can write $H^d$ as a
Haar shift operator of index $\tau=1$ with $a_{I',I''}=\pm 1$ for
$I'=I$, $I''=I_{\pm}$ and $a_{I',I''}=0$ otherwise. These are the
operators used by Petermichl~\cite{MR1756958,MR2354322} to
approximate the Hilbert transform.  More precisely, she used the
family of operators $H^d_{s,t}$, $s\in \re$, $t>0$, which are
defined as above but with the dyadic grid replaced by its
translation by $s$ and dilation by $t$.    The Hilbert transform is
then the limit of integral averages of these operators, so norm
inequalities for $H$ follow from norm inequalities for $H^d_{s,t}$
by Fatou's lemma and Minkowski's inequality.    Similar
approximations hold for the Riesz transforms and Beurling-Ahlfors
operator, and we refer the reader to \cite{MR2367098,MR1894362} for
more details.

\bigskip

To apply Theorem~\ref{thm:lerner} to the Haar shift operators we
need two lemmas.  The first is simply that CZ Haar shift operators
satisfy a weak $(1,1)$ inequality.  The proof of this is known but
we could not find it in the literature, except when $\tau=0$---in this
case $H_\tau$ is a constant Haar multiplier and the proof is given
in~\cite{MR1864538}.  Therefore, we provide a brief sketch
of the details.  Here and below we will use the following notation:
given an integer $\tau\ge 0$ and a dyadic cube $Q$, let $Q^\tau$
denote its $\tau$-th generation ``ancestor'': that is, the unique
dyadic cube $Q^\tau$ containing $Q$ such that $|Q^\tau|=2^{\tau
n}|Q|$.

\begin{lemma}
Given an integer $\tau\geq 0$,  there exists a constant $C_{\tau,n}$
such that for every $t>0$,
\[|\{ x \in \re^n: |H_\tau f(x)|> t \}| \leq \frac{C_{\tau,n}}{t}
\int_{\re^n} |f(x)|\,dx. \]
\end{lemma}

\begin{proof}
Fix $t>0$ and form the Calder\'on-Zygmund decomposition of $f$ at
height $t$.  Decompose $f$ as  the sum of the good and bad parts,
$g+b$.  The estimate for $g$ is standard.  For $b$, since
Lebesgue measure is doubling, it suffices to show that the set
\[ |\{ x \in \re^n\setminus (\cup_j Q_j^\tau): |H_\tau b(x)|> t/2 \}| \]
has measure 0. Fix $j$ and $x\in \re^n\setminus Q_j^\tau$; then we would be done if
we could show that $H_\tau b_j(x)=0$.   Fix a term $a_{Q',Q''} \la
b_j, h_{Q'} \rangle h_{Q''}(x)$ in the sum defining $H_\tau b_j(x)$.
If this is non-zero, then $h_{Q''}(x)\neq 0$, so $Q''\cap
\re^n\setminus Q_j^\tau \neq \emptyset$.  Since $Q''\subset Q$,
$Q\cap \re^n\setminus Q_j^\tau \neq \emptyset$.  On the other hand,
since $\int_{Q_j} b_j(x)\,dx=0$, $\la b_j, h_{Q'} \rangle \neq 0$
only if $Q'\subset Q_j$, which in turn implies that $Q\subset
Q_j^\tau$, a contradiction.
\end{proof}

Our second lemma is a key estimate that is a sharper variant of a
result known for Calder\'on-Zygmund singular integrals (see
\cite{MR779906}) and whose proof is similar.  For completeness we
include the details.

\begin{lemma} \label{lemma:osc-est}
Given $\tau \geq 0$, let $H_\tau$ be a CZ Haar shift operator. Fix
$\lambda$,  $0<\lambda <1 $. Then for any function $f$,  every
dyadic cube $Q_0$, and every $x\in Q_0$,
\begin{gather*}
\omega_\lambda (H_\tau f,Q_0) \leq
\frac{C_{\tau,n}}{\lambda}\,\avgint_{Q_0^\tau}|f(x)|\,dx, \\
M^{\#,d}_{\lambda,Q_0}(H_\tau f)(x) \leq \frac{C_{\tau,n}}{\lambda} \,
M^df(x).
\end{gather*}
\end{lemma}

\begin{proof}
It suffices to prove the first inequality; the second follows
immediately from the definition of $M^{\#,d}_{\lambda,Q_0}$. Fix $Q_0$
and write $H_\tau$ as the sum of two operators:
\[ H_\tau f(x) = H_\tau(f\chi_{Q_0^\tau})(x) + H_\tau(f\chi_{\re^n\setminus Q_0^\tau})(x). \]
We claim the second term is constant for all $x\in Q_0$.  Let $Q$ be
any dyadic cube.  Then the corresponding term in the sum defining
$H_\tau(f\chi_{\re^n\setminus Q_0^\tau})(x)$ is
\begin{equation} \label{eqn:osc-est1}
\sum_{\substack{Q',Q''\in\Delta(Q)\\ 2^{-\tau n}|Q|\leq |Q'|,|Q''|}}
a_{Q',Q''} \la f\chi_{\re^n\setminus Q_0^\tau}, h_{Q'} \rangle
h_{Q''}(x).
\end{equation}
We may assume that $Q''\cap Q_0\neq \emptyset$ (otherwise we get a
zero term); since $Q''\subset Q$, this implies that $Q\cap
Q_0^\tau\neq \emptyset$.   Similarly, we have $Q\cap (\re^n\setminus
Q_0^\tau)\neq \emptyset$.  Therefore, $Q_0^\tau \subsetneq Q$, so
$|Q_0| <2^{-\tau n}|Q|\leq |Q''|$.  Hence, $Q_0\subsetneq Q''$ and
$h_{Q''}$ is constant on $Q_0$. Thus, \eqref{eqn:osc-est1} does not
depend on $x$ and so is constant on $Q_0$.

Denote this constant by $H_\tau f(Q_0)$; then
\[
|\{ x \in Q_0 : |H_\tau f(x)- H_\tau f(Q_0)|> t \}| = |\{ x \in Q_0
: |H_\tau(f\chi_{Q_0^\tau})(x)|>t \}|.
\]
Since $H_\tau$ is a CZ Haar shift operator it is weak $(1,1)$.
Therefore, by inequality~\eqref{eqn:mean-est1},
\begin{multline*}
\omega_\lambda (H_\tau f, Q_0)   \leq \big((H_\tau f- H_\tau
f(Q_0))\chi_{Q_0}\big)^*(\lambda |Q_0|) \\
\leq
\lambda^{-1}\|H_\tau(f\chi_{Q_0^\tau})\|_{L^{1,\infty}(Q_0,|Q_0|^{-1}dx)}
\leq \frac{C_{\tau,n}}{\lambda} \avgint_{Q_0^\tau} |f(x)|\,dx.
\end{multline*}
\end{proof}

\section{Singular integrals, paraproducts and Haar multipliers}
\label{section:main-proof}

In this section we prove Theorems~\ref{thm:sio-onewt},
\ref{thm:sio-twowt}, \ref{thm:paraproduct} and \ref{thm:haarmult}.
The principal results are the first two for singular integrals; the
results for paraproducts and constant Haar multipliers are
variations of these and we will only sketch the changes.   We will
also indicate the technical obstacles in attempting to apply our
results to general Calder\'on-Zygmund singular integrals.

\subsection*{One weight inequalities: proof of Theorem~\ref{thm:sio-onewt}}

As we discussed in the previous section, to prove
Theorem~\ref{thm:sio-onewt} it will suffice to establish the analogous
result for Haar shift operators.

\begin{theorem} \label{thm:haarshift-onewt}
Given an integer $\tau \geq 0$ and a CZ Haar shift operator
$H_\tau$, and  given $p$, $1<p<\infty$, then for any $w\in A_p$,
\[ \|H_\tau f\|_{L^p(w)} \leq C_{\tau, n,p} \,[w]_{A_p}^{\max\left(1,\frac{1}{p-1}\right)}\|f\|_{L^p(w)}. \]
\end{theorem}

\begin{proof}[Proof of Theorem \ref{thm:haarshift-onewt}]
By Theorem~\ref{thm:sharp-extrapol} it will suffice to prove that
\[ \|H_\tau f\|_{L^2(w)} \leq C_{\tau, n} [w]_{A_2}\|f\|_{L^2(w)}. \]
Fix $w\in A_2$ and fix $f$.  By a standard approximation argument we
may assume without loss of generality that $f$ is bounded and has
compact support.  Let $\re^n_j$, $1\le j\le 2^n$, denote the
$n$-dimensional quadrants in $\re^n$: that is, the sets
$I^{\pm}\times I^{\pm}\times\cdots\times I^{\pm}$ where $I^+ =
[0,\infty)$ and $I^-=(-\infty,0)$.

For each $j$, $1\le j\le 2^n$, and for each $N>0$ let $Q_{N,j}$ be
the dyadic cube adjacent to the origin of side length $2^N$ that is
contained in $\re^n_j$. Since  $Q_{N,j}\in \Delta$,
$\Delta(Q_{N,j})\subset \Delta$. Because $H^\tau$ is a CZ shift operator,
it is bounded on $L^2$.    Thus, since $f\in L^2$,  by \eqref{eqn:median-f*} and
\eqref{eqn:mean-est2}, $m_{H_\tau f} (Q_{N,j})\rightarrow 0$ as
$N\rightarrow \infty$.   Therefore, by Fatou's lemma and Minkowski's
inequality,
\[ \|H_\tau f\|_{L^2(w)} \leq \liminf_{N\rightarrow \infty}
\sum_{j=1}^{2^n} \left(\int_{Q_{N,j}} |H_\tau f(x) - m_{H_\tau
f}(Q_{N,j})|^2w(x)\,dx\right)^{1/2}. \]
Hence, it will suffice to prove that each term in the sum on the
right is bounded by $C_{\tau,n}[w]_{A_2}\|f\|_{L^2(w)}$.

Fix $j$ and let $Q_N=Q_{N,j}$.  By Theorem~\ref{thm:lerner} and
Lemma~\ref{lemma:osc-est}, for every $x\in Q_N$ we have that
\begin{align}\label{eqn:Htau-decomp}
& |H_\tau f(x) - m_{H_\tau f}(Q_N)|
\\ \nonumber
&\qquad \qquad \le 4\,M^{\#,d}_{\frac{1}{4},Q_N}(H_\tau f)(x) +
4\sum_{j,k} \omega_{\frac{1}{2^{n+2}}}(H_\tau
f,\widehat{Q}_j^k)\chi_{Q_j^k}(x)
\\ \nonumber
&\qquad \qquad \le C_{\tau,n}\, M f(x) + C_{\tau,n}\,\sum_{j,k}
\left(\avgint_{P_j^k} |f(x)|\,dx\right)\, \chi_{Q_j^k}(x)
\\ \nonumber
& \qquad \qquad = C_{\tau,n}\, M f(x) + C_{\tau,n}\, F(x),
\end{align}
where $P_j^k=(\widehat{Q}_j^k)^\tau$. We get the desired estimate
for the first term from Theorem~\ref{thm:sharpmax} with $p=2$:
\begin{align*}
\|M f\|_{L^2(Q_N,w)} \le \|M f\|_{L^2(w)} \leq C_{n}\,[w]_{A_2}
\|f\|_{L^2(w)}.
\end{align*}

\medskip

To estimate $F$ we use duality.  Fix a non-negative function $h\in
L^2(w)$ with $\|h\|_{L^2(w)}=1$; then by Remark~\ref{rem:disjoint}
and Lemma~\ref{lemma:wtdmax} we have that
\begin{align*}
\int_{Q_N} F(x)\,h(x)\,w(x)\,dx & = C_{\tau,n}\, \sum_{j,k}
\avgint_{P_j^k} |f(x)|\,dx
\int_{Q_j^k} w(x)h(x)\,dx  \\
& \leq  2 \cdot 2^{(\tau+1)n}\sum_{j,k}
\frac{w(P_j^k)}{|P_j^k|}\frac{w^{-1}(P_j^k)}{|P_j^k|} \; |E_j^k|\\
& \qquad \times \frac{1}{w^{-1}(P_j^k)} \int_{P_j^k} |f(x)|w(x)
w(x)^{-1}\,dx \\
& \qquad \times \frac{1}{w(Q_j^k)}\int_{Q_j^k} h(x)w(x)\,dx \\
& \leq C_{\tau,n}\, [w]_{A_2}\sum_{j,k} \int_{E_j^k} M^d_{w^{-1}}(fw)(x) M^d_w h(x)\,dx \\
& \leq C_{\tau,n}\, [w]_{A_2}\int_\subRn  M^d_{w^{-1}} (fw)(x)
M^d_w h(x)\,dx \\
& \leq C_{\tau,n}\,[w]_{A_2}
\left(\int_\subRn M^d_{w^{-1}} (fw)(x)^2 w(x)^{-1}\,dx\right)^{1/2} \\
& \qquad \qquad \times \left(\int_\subRn M^d_{w} h(x)^2 w(x)\,dx\right)^{1/2} \\
& \leq C_{\tau,n}\,[w]_{A_2}
\left(\int_\subRn |f(x)w(x)|^2 w(x)^{-1}\,dx\right)^{1/2} \\
& \qquad \qquad \times \left(\int_\subRn h(x)^2 w(x)\,dx\right)^{1/2} \\
& = C_{\tau,n}\,[w]_{A_2} \left(\int_\subRn
|f(x)|^2w(x)\,dx\right)^{1/2}.
\end{align*}
If we take the supremum over
all such functions $h$, we conclude that
$$
\|F\|_{L^2(Q_N,w)} \le C_{\tau,n}\,[w]_{A_2}\,\|f\|_{L^2(w)}.
$$
Combining our estimates we have that
\[
\left(\int_{Q_N} |H_\tau f(x)-m_{H_\tau f}(Q_N)|^2
w(x)\,dx\right)^{1/2} \leq C_{\tau,n}\,[w]_{A_2} \|f\|_{L^2(w)},
\]
and this completes the proof.
\end{proof}

\subsection*{Two weight inequalities:  Proof of Theorem~\ref{thm:sio-twowt}}

To prove Theorem~\ref{thm:sio-twowt} it will suffice to establish the
corresponding result for the Haar shift operators.  We record this
as a separate result.

\begin{theorem} \label{thm:haarshift-twowt}
Given an integer $\tau \geq 0$, let $H_\tau$ be a CZ Haar shift
operator.  Given $p$, $1<p<\infty$, and let $A$ and $B$ be Young
functions such that $\bar{A}\in B_{p'}$ and $\bar{B}\in B_p$. Then
for any pair $(u,v)$ such that \eqref{eqn:sio-twowt1} holds we have
that
\[ \|H_\tau f\|_{L^p(u)} \leq C\|f\|_{L^p(v)}. \]
\end{theorem}

\begin{proof}[Proof of Theorem \ref{thm:haarshift-twowt}]

The proof is very similar to the proof of
Theorem~\ref{thm:haarshift-onewt} replacing the $A_2$ estimates with
an argument from the second half of the proof of the main theorem in
\cite{MR2351373}; therefore, we omit many of the details.

We argue as in the one-weight case and with the same notation;  it
will suffice to prove
\[ \left(\int_{Q_N} |H_\tau f(x)-m_{H_\tau f}(Q_N)|^p
 u(x)\,dx\right)^{1/p} \leq C\|f\|_{L^p(v)} \]
and we use \eqref{eqn:Htau-decomp}.  The estimate of  the term
containing the maximal operator is straightforward:  by
Theorem~\ref{thm:twowt-max} we have $M: L^p(v)\rightarrow L^p(u)$
since the pair $(u,v)$ satisfies \eqref{eqn:sio-twowt1}. Therefore,
by duality (this time with respect to Lebesgue measure) it is
enough to show that for every non-negative $h\in L^{p'}(\re^n)$ with
$\|h\|_{L^{p'}}=1$,
\[ I=\int_{Q_N} F(x)\,u(x)^{1/p} \,h(x)\,dx\leq C\|f\|_{L^p(v)}. \]
We apply Remark~\ref{rem:disjoint} and the generalized H\"older's
inequality to get
\begin{align*}
I & \leq C\sum_{j,k} \avgint_{P_j^k} |f(x)|\,dx \avgint_{Q_j^k}
u(x)^{1/p}h(x)\,dx \; |E_j^k| \\
&\leq C\sum_{j,k} \|fv^{1/p}\|_{\bar{B},P_j^k}\|v^{-1/p}\|_{B,P_j^k}
\|u^{1/p}\|_{A,Q_j^k}\|h\|_{\bar{A},Q_j^k}\;|E_j^k|.
\end{align*}
By convexity, $\|u^{1/p}\|_{A,Q_j^k}\leq
 2^{n(\tau+1)}\|u^{1/p}\|_{A,P_j^k}$, so since the pair $(u,v)$ satisfies
 \eqref{eqn:sio-twowt1},
\begin{align*}
I & \leq C\,\sum_{j,k} \int_{E_j^k}
M_{\bar{B}}(fv^{1/p})(x)M_{\bar{A}}h(x)\,dx \\
& \leq C\,\int_{\re^n} M_{\bar{B}}(fv^{1/p})(x)M_{\bar{A}}h(x)\,dx.
\end{align*}
Since $\bar{A}\in B_{p'}$ and $\bar{B}\in B_p$, by
Theorem~\ref{thm:orlicz-max}, $M_{\bar{B}}$ is bounded on $L^p$ and
$M_{\bar{A}}$ is bounded in $L^{p'}$.  The desired estimate now
follows by H\"older's inequality.
\end{proof}

\subsection*{General Calder\'on-Zygmund singular integrals}
Key to the proofs of
Theorems~\ref{thm:haarshift-onewt} and~\ref{thm:haarshift-twowt}
are the sharp estimates for the local mean oscillation in
Lemma~\ref{lemma:osc-est}.  If we were to try to extend these proofs
to an arbitrary Calder\'on-Zygmund singular integral $T$, then we
would have to estimate the local mean oscillation by a sum (see
\cite{MR779906}):
\[ \omega_\lambda (Tf,Q) \leq C\sum_{i=0}^\infty 2^{-i} \avgint_{2^i Q}
|f(x)|\,dx. \]

If we use this estimate in the proof of
Theorem~\ref{thm:haarshift-onewt}, then we still get that $T$ is
bounded (since the sum is bounded by $2\inf_{x\in Q} Mf(x)$), but we
get an additional factor of $[w]_{A_2}$.   The proof of
Theorem~\ref{thm:haarshift-twowt} can be modified to handle this
sum, but to get convergence you need the additional assumption that
$p>n$.  This is the approach used by  Lerner~\cite{lernerP2009}.
This (seemingly artificial) restriction $p>n$ also appears in
\cite{MR2351373}.  It would be very interesting to find a refinement
of Theorem~\ref{thm:lerner} that would let us remove this
restriction.  Alternatively, it is tempting to conjecture that the
estimate above could be improved by replacing $2^{-i}$ by
$2^{-(n+\epsilon)i}$, which would be sufficient to adapt the proofs
in both the one and two-weight case.  However, it is not clear that
such an inequality is true, even for singular integrals with smooth
kernels.

\subsection*{Dyadic paraproducts and Haar multipliers:  Proof of Theorems~\ref{thm:paraproduct} and \ref{thm:haarmult}}

The proof of Theorem~\ref{thm:paraproduct} is essentially identical
to the proof of the corresponding results for singular integrals
once we prove the analog of Lemma~\ref{lemma:osc-est}:
\[ \omega_\lambda(\pi_b f,Q) \leq \frac{C\|b\|_{*,d}}{\lambda}\avgint_Q |f(x)|\,dx. \]

The proof follows as before.   The dyadic paraproduct is a local
operator, since  for any $I\in\Delta$, $h_I$ is constant on proper
dyadic sub-intervals of $I$, and so, given a fixed dyadic interval
$I_0$, $\pi_b(f\chi_{\re\setminus I_0})$ is constant on $I_0$.
Furthermore, $\pi_b$ is bounded on $L^p$ and satisfies a weak
$(1,1)$ inequality:  for every $t>0$,
\[ |\{ x \in \re : |\pi_bf(x)|> t\}| \leq
\frac{C\|b\|_{*,d}}{t}\int_{\re} |f(x)|\,dx.  \]
For a proof, see Pereyra~\cite{MR1864538}.

\medskip

Theorem~\ref{thm:haarmult} is actually a special case of
Theorems~\ref{thm:haarshift-onewt} and~\ref{thm:haarshift-twowt},
since the constant Haar multipliers are clearly Haar shift operators
of index  $\tau=0$. The dependence on $\|\alpha\|_{\ell^\infty}$
follows at once by linearity.

\section{Maximal singular integrals}
\label{section:maximal}

In this section we prove Theorem~\ref{thm:max-singular}.  To do so,
we will follow the approach used by Hyt\"onen {\em et
al.}~\cite{hytonen-lacey-reguera-vagharshakyanP2009}  and actually
prove the corresponding result for a family of ``maximal'' dyadic
shift operators.  The underlying dyadic operators are a
generalization of the Haar shift operators defined in
Section~\ref{section:dyadic}.  As noted
in~\cite{hytonen-lacey-reguera-vagharshakyanP2009}, the results for
the maximal singular integrals associated to the Hilbert transform,
the Riesz transforms and the Beurling-Ahlfors operator are gotten by
the same approximation arguments as we discussed above.

We begin by defining the appropriate shift operators.  To
distinguish them from the operators defined above, we will refer to
them as generalized Haar shift operators.  (In
\cite{hytonen-lacey-reguera-vagharshakyanP2009} they are simply
referred to as Haar shift operators, but our change in terminology
should not cause any confusion.)  We say that an operator $T$ is a
generalized Haar shift operator of index $\tau\geq 0$ if
\[ T f=\sum_{Q\in\Delta} \langle f,g_Q\rangle\,\gamma_Q, \]
where the functions $\gamma_Q$ are such that:
\begin{enumerate}
\item $\supp(\gamma_Q) \subset Q$;
\item if $Q'\in \Delta$ and $Q'\subset Q$ with $|Q'|\le 2^{-\tau\,n}\,|Q|$, then $\gamma_Q$ is constant on $Q'$;
\item $\|\gamma_Q\|_\infty \leq |Q|^{-1/2}$.
\end{enumerate}
The functions $g_Q$ also have these properties. Finally we assume
that the functions $\gamma_Q,\,g_Q$ are such that $T$ extends to a
bounded operator on $L^2$.  Together, these hypotheses imply that
$T$ is of weak-type $(1,1)$ (see
\cite{hytonen-lacey-reguera-vagharshakyanP2009}).  Examples of
generalized Haar shift operators include the dyadic paraproducts and
their adjoints.

Associated with a generalized Haar shift operator $T$ is the maximal
Haar shift operator
\[ T_* f
= \sup_{\epsilon>0}|T_\epsilon f|
=\sup_{\epsilon>0}\bigg|\sum_{\substack{Q\in\Delta\\ |Q|\ge
\epsilon^n}} \langle f,g_Q\rangle\,\gamma_Q\bigg|, \]
We again have that $T_*$ is bounded on $L^2$ and is of weak-type
$(1,1)$ (see \cite{hytonen-lacey-reguera-vagharshakyanP2009}).

Our main result for maximal Haar shift operators is the following.

\begin{theorem} \label{thm:T*}
Let $T$ be a generalized Haar shift operator of index $\tau\geq 0$,
and let $T_*$ be the corresponding maximal Haar shift operator.
Then, for every  $p$, $1<p<\infty$, and for all $w\in A_p$,
\[ \|T_* f\|_{L^p(w)} \leq
C_{\tau,n,p}
[w]_{A_p}^{\max\left(1,\frac{1}{p-1}\right)}\|f\|_{L^p(w)}. \]
Furthermore, if the pair of weights $(u,v)$ satisfies
\eqref{eqn:sio-twowt1}, then
\[ \|T_* f\|_{L^p(u)} \leq
C \|f\|_{L^p(v)}. \]
\end{theorem}

The proof of Theorem~\ref{thm:T*} is very much the same as the
proofs of Theorem~\ref{thm:haarshift-onewt}
and~\ref{thm:haarshift-twowt}, so we will only describe the
differences between the two arguments. First, if $f$ is bounded and
has compact support, $\sup_{\epsilon>0}|m_{T_\epsilon
f}(Q)|\rightarrow 0$ as $|Q|\to\infty$. Indeed, by
\eqref{eqn:median-f*} and \eqref{eqn:mean-est2},
$$
\sup_{\epsilon>0}|m_{T_\epsilon f}(Q)| \le 2^{1/p} \sup_{\epsilon>0}
\left(\frac1{|Q|}\int_Q |T_\epsilon f|^2\,dx\right)^{1/2} \le
|Q|^{-1/2}\,\|T_* f\|_{L^2},
$$
and the right-hand term tends to $0$ as $|Q|\to\infty$ since $T_*$
is bounded on $L^2$.  Now fix $j$,  $1\le j\le n$,  and as before
let $Q_N=Q_{N,j}$.   Then in the one-weight case by Fatou's lemma we
have that
\begin{align*}
& \int_{\re^n_j} |T_* f(x)|^2w(x)\,dx \\
& \qquad \qquad \le \liminf_{N\to\infty} \int_{Q_N}
\big|\sup_{\epsilon>0}|T_\epsilon f(x)|-
\sup_{\epsilon>0}|m_{T_\epsilon f}(Q_N)| \big|^2w(x)\,dx
\\
& \qquad \qquad \le \liminf_{N\to\infty} \int_{Q_N}
\sup_{\epsilon>0}|T_\epsilon f(x)- m_{T_\epsilon f}(Q_N)
\big|^2w(x)\,dx.
\end{align*}
In the two-weight case we get the same inequality with $2$ replaced
by $p$ and $w$ with the weight $u$.

Fix $\epsilon>0$ and apply Theorem~\ref{thm:lerner}.   To continue
the proof we need an analog of Lemma~\ref{lemma:osc-est} that takes
into account the supremum.  This in turn reduces to showing the
following:  given  $\lambda$, $0<\lambda<1$, and a dyadic cube
$Q_0$, for every $x\in Q_0$
\begin{equation}\label{eqn:osc-Teps}
\sup_{\epsilon>0}\omega_\lambda (T_\epsilon f,Q_0) \leq
\frac{C_{\tau,n}}{\lambda}\,\avgint_{Q_0^\tau}|f(x)|\,dx.
\end{equation}
Given inequality~(\ref{eqn:osc-Teps}) the remainder of the proof in
both the one and two-weight case proceeds exactly as before.

\medskip

To prove \eqref{eqn:osc-Teps} we proceed as in the proof of
Lemma~\ref{lemma:osc-est}. Fix $\epsilon>0$; since  $T_\epsilon$ is
linear,
$$
T_\epsilon f(x) = T_\epsilon(f\chi_{Q_0^\tau})(x) +
T_\epsilon(f\chi_{\re^n\setminus Q_0^\tau})(x)
$$
We claim that the second term is constant for all $x\in Q_0$; denote
it by $T_\epsilon (Q_0)$.  Assuming this for the moment, we have
that
\begin{align*}
|\{ x \in Q_0 : |T_\epsilon f(x)- T_\epsilon f(Q_0)|> t \}| & = |\{
x \in Q_0 : |T_\epsilon(f\chi_{Q_0^\tau})(x)|>t \}|
\\
& \le
|\{ x \in Q_0 : T_*(f\chi_{Q_0^\tau})(x)>t \}| \\
& \le \frac{C}{t}\,\int_{Q_0^\tau}|f(x)|\,dx,
\end{align*}
where we have used that $T_*$ is of weak-type $(1,1)$. Inequality
(\ref{eqn:osc-Teps}) follows at once from this and
\eqref{eqn:mean-est1}:
\begin{multline*}
\sup_{\epsilon>0} \omega_\lambda (T_\epsilon f, Q_0)\\
\leq \sup_{\epsilon>0}\big((T_\epsilon f- T_\epsilon
f(Q_0))\chi_{Q_0}\big)^*(\lambda |Q_0|) \leq
\frac{C_{\tau,n}}{\lambda} \avgint_{Q_0^\tau} |f(x)|\,dx.
\end{multline*}

It remains to show that $T_\epsilon f(Q_0)$ is indeed a constant.
Fix $x\in Q_0$; then
$$
T_\epsilon(f\chi_{\re^n\setminus Q_0^\tau})(x) =
\sum_{\substack{Q\in\Delta\\|Q|\ge \epsilon^n}} \langle f\chi_{\re^n\setminus Q_0^\tau},
g_Q\rangle \gamma_Q(x).
$$
We may restrict the sum to those cubes $Q$ satisfying $Q\cap Q_0\neq
\emptyset$ and $Q\cap (\re^n\setminus Q_0^\tau)\neq\emptyset$  since
otherwise we get  terms equal to $0$. In this case,
$Q_0^\tau\subsetneq Q$, and consequently $Q_0\subset Q$ with
$|Q_0|<2^{-\tau\,n}\,|Q|$. This implies that $\gamma_Q$ is constant
on $Q_0$ which proves our claim.

\section{The dyadic square function}
\label{section:Sd}

To prove our results for the dyadic square function we must first
give a version of Lemma~\ref{lemma:osc-est}.  The key change,
however, is that we prove it not for $S_d f$ but for $(S_d f)^2$.

\begin{lemma} \label{lemma:osc-est-Sd}
Fix $\lambda$, $0<\lambda < 1$.  Then for any function $f$,
every dyadic cube $Q_0$, and every $x\in Q_0$,
\begin{gather*}
\omega_\lambda ((S_d f)^2,Q_0) \leq
\frac{C_n}{\lambda^2}\left(\avgint_{Q_0}|f(x)|\,dx\right)^2, \\
M^{\#,d}_{\lambda,Q_0}((S_d f)^2)(x) \leq \frac{C_n}{\lambda^2} M^df(x)^2.
\end{gather*}
\end{lemma}

\begin{proof}
It suffices to prove the first inequality; the second follows
immediately from the definition of $M^{\#,d}_{\lambda,Q_0}$. Fix
$Q_0$; then  for every $x\in Q_0$ we can decompose $S_d f(x)^2$ as
\[ S_d f(x)^2 =
\sum_{\substack{Q\in\Delta \\ Q\subsetneq Q_0}}
|f_Q-f_{\widehat{Q}}|^2\chi_{Q}(x) +
\sum_{\substack{Q\in\Delta \\ Q\supset Q_0}}
|f_Q-f_{\widehat{Q}}|^2.
\]
The second term is a constant;  denote it by $S_d f(Q_0)^2$.
Furthermore, we have that for $x\in Q_0$,
$$
0\le S_d f(x)^2-S_d f(Q_0)^2=
\sum_{\substack{Q\in\Delta \\Q\subsetneq Q_0} }
|f_Q-f_{\widehat{Q}}|^2\chi_{Q}(x) \le S_d (f\chi_{Q_0})(x)^2.
$$
Hence, since  $S_d$ is  weak $(1,1)$ (see, for instance, \cite{wilson07}), for
every $t>0$ we have that
\begin{multline*}
|\{ x \in Q_0 : |S_d f(x)^2-S_d f(Q_0)^2|> t \}| \\
\le |\{ x \in Q_0 : |S_d (f\chi_{Q_0})(x)|>t^{1/2} \}| \leq
\frac{C_n}{t^{1/2}} \int_{Q_0} |f(x)|\,dx.
\end{multline*}
Therefore, by \eqref{eqn:mean-est1} with $p=1/2$,
\begin{multline*}
\omega_\lambda ((S_d f)^2, Q_0)  \\ \leq \big(((S_d f)^2- S_d
f(Q_0)^2)\chi_{Q_0}\big)^*(\lambda |Q_0|) \leq \frac{C_n}{\lambda^2}
\left(\avgint_{Q_0} |f(x)|\,dx\right)^2.
\end{multline*}
\end{proof}

\subsection*{One weight inequalities:  Proof of Theorem~\ref{thm:square}}
The proof is a variation of the proof of
Theorem~\ref{thm:haarshift-onewt} and we describe the main changes.
The first is that rather than proving this result for $p=2$, we
choose $p=3$, so that $1/2= (p-1)^{-1}$.   Then by
Theorem~\ref{thm:sharp-extrapol} it will suffice to prove that for
any $w\in A_3$,
\[ \|S_d f\|_{L^3(w)} \leq C_n[w]_{A_3}^{\frac12}\|f\|_{L^3(w)}. \]

Fix $w\in A_3$ and $j$, $1\leq j \leq 2^n$. As before, let
$Q_N=Q_{N,j}$. Then,
\begin{multline*}
 \left(\int_{\re^n_j} |S_df(x)|^3w(x)\,dx\right)^{2/3} \\
\leq \liminf_{N\rightarrow\infty} \left(\int_{Q_N} |S_d f(x)^2 -
m_{(S_d f)^2}(Q_N)|^{3/2} w(x)\,dx\right)^{2/3}.
\end{multline*}
By Theorem~\ref{thm:lerner} and Lemma~\ref{lemma:osc-est-Sd}, for
every $x\in Q_N$ we have
\begin{align}\label{eqn:Sd-decomp}
&|S_d f(x)^2 - m_{(S_d f)^2}(Q_N)|
\\ \nonumber
&\qquad\quad \le C_{n}\, Mf(x)^2+C_n\,\sum_{j,k}
\left(\avgint_{\widehat{Q}_j^k} |f(x)|\,dx\right)^2\,\chi_{Q_j^k}(x)
\\ \nonumber
&\qquad\quad = C_{n}\, Mf(x)^2+C_n\,F(x).
\end{align}
To estimate the first term we use Theorem~\ref{thm:sharpmax} with
$p=3$:
$$
\|(M f)^2\|_{L^{3/2}(Q_N,w)} \le \|M f\|_{L^3(w)}^2 \le C_n
\,[w]_{A_3} \|f\|_{L^3(w)}^2.
$$
To estimate $F$ we use duality.  Fix a non-negative function $h\in L^{3}(w)$
with $\|h\|_{L^3(w)}=1$; then Remark~\ref{rem:disjoint} and
Lemma~\ref{lemma:wtdmax} yield
\begin{align*}
\int_{Q_N} F(x)\,h(x)\,dx & = \sum_{j,k}
\left(\avgint_{\widehat{Q}_j^k} |f(x)|\,dx\right)^2
\int_{Q_j^k} w(x)h(x)\,dx  \\
& \leq 2^{n+1}\sum_{j,k}
\frac{w(\widehat{Q}_j^k)}{|\widehat{Q}_j^k|}\left(\frac{w^{-1/2}(\widehat{Q}_j^k)}{|\widehat{Q}_j^k|}\right)^2 \; |E_j^k|\\
& \qquad \times \left(\frac{1}{w^{-1/2}(\widehat{Q}_j^k)}
\int_{\widehat{Q}_j^k} |f(x)|w(x)^{1/2}
w(x)^{-1/2}\,dx\right)^2 \\
& \qquad \times \frac{1}{w(Q_j^k)}\int_{Q_j^k} h(x)w(x)\,dx \\
& \leq C_n [w]_{A_3}\sum_{j,k} \int_{E_j^k} M^d_{w^{-1/2}}(fw^{1/2})(x)^2 M^d_w h(x)\,dx \\
& \leq C_n [w]_{A_3}\int_\subRn  M^d_{w^{-1/2}} (fw^{1/2})(x)^2
M^d_w h(x)\,dx
\\
& \leq C_n[w]_{A_3}
\left(\int_\subRn M^d_{w^{-1/2}} (fw^{1/2})(x)^3 w(x)^{-1/2}\,dx\right)^{2/3} \\
& \qquad \qquad \times \left(\int_\subRn M^d_{w} h(x)^3 w(x)\,dx\right)^{1/3} \\
& \leq C_n[w]_{A_3}
\left(\int_\subRn |f(x)w(x)^{1/2}|^3 w(x)^{-1/2}\,dx\right)^{2/3} \\
& \qquad \qquad \times \left(\int_\subRn h(x)^3 w(x)\,dx\right)^{1/3} \\
& = C_n[w]_{A_3} \|f\|_{L^3(w)}^2.
\end{align*}
Taking the supremum over all such functions $h$ we conclude that
$$
\|F\|_{L^{3/2}(Q_N,w)} \le C_n[w]_{A_3} \|f\|_{L^3(w)}^2.
$$
If we combine the two estimates we get
$$
\left(\int_{Q_N} |S_d f(x)^2 - m_{(S_d f)^2}(Q_N)|^{3/2}
w(x)\,dx\right)^{1/3} \le C_n\,[w]_{A_3}^{1/2} \|f\|_{L^3(w)},
$$
and the desired inequality follows as before.

\medskip

The exponent $\max\big(\frac{1}{2},\frac{1}{p-1}\big)$ is the best
possible.  As we noted above, for $p\leq 2$ specific examples were
constructed by Dragi{\v{c}}evi{\'c} {\em et al.}~\cite{MR2140200}.
For $p>2$, a proof was sketched by Lerner~\cite{MR2200743}, adapting
an argument for singular integrals due to R.~Fefferman
and Pipher~\cite{fefferman-pipher97}.  For completeness we give the
details.

If $2<p\leq 3$, then the sharpness of this exponent follows at once by
extrapolation.  For suppose there existed $p_0$ in this range such
that the best possible exponent satisfied
$\alpha(p_0)<\frac{1}{p_0-1}$.  Then by
Theorem~\ref{thm:sharp-extrapol}, we get that the exponent in the
weighted $L^2$ inequality is
\[ \alpha(p_0)\max\left(1,\frac{p_0-1}{2-1}\right)< 1, \]
contradicting the fact that the best possible exponent is
$1$.

We now consider the case $p>3$.
Suppose to the contrary that there exists a non-decreasing function
$\phi$ such that $\phi(t)/t^{1/2} \rightarrow 0$ as
$t\rightarrow\infty$, and suppose that for some $p_0>2$,
\begin{equation} \label{eqn:sharp1}
 \|S_d f\|_{L^{p_0}(w)} \leq C_{n,p_0} \phi([w]_{A_{p_0}})\|f\|_{L^{p_0}(w)}.
\end{equation}
We will show that this implies for all $p>p_0$ that
\begin{equation} \label{eqn:sharp2}
 \|S_d f\|_{L^{p}} \leq C_1 \phi(C_2p)\|f\|_{L^{p}}.
\end{equation}
Below we will give an example to show that this is a contradiction.

To prove \eqref{eqn:sharp2}, fix $p>p_0$ and fix a non-negative function $h\in L^{(p/p_0)'}$,
${\|h\|_{L^{(p/p_0)'}(\re^n) }=1}$.   Define the Rubio de
Francia iteration algorithm (see~\cite{cruz-martell-perezBook})
\[ Rh = \sum_{k=0}^\infty \frac{M^k h}{2^k \,{\|M\|^k_{L^{(p/p_0)'}(\re^n) }}   }.  \]
Then it follows from this definition that $
{\|Rh\|_{L^{(p/p_0)'}(\re^n) }} \leq 2$ and
\[ [Rh]_{A_{p_0}} \leq [Rh]_{A_1} \leq 2\,{\|M\|_{L^{(p/p_0)'}(\re^n) }} \leq C_{n,p_0}p.  \]
Therefore, by \eqref{eqn:sharp1} and H\"older's inequality,
\begin{multline*}
\int_{\re^n} S_df(x)^{p_0} h(x)\,dx \leq \int_{\re^n} S_df(x)^{p_0}
Rh(x)\,dx \\ \leq C_{n,p_0}\phi([Rh]_{A_{p_0}})^{p_0} \int_{\re^n}
f(x)^{p_0}Rh(x)\,dx \leq C_{n,p_0}\phi(C_{n,p_0} p)^{p_0}
{\|f\|_{L^p(\re^n)}^{p_0}}.\!
\end{multline*}
Inequality \eqref{eqn:sharp2} now follows by duality, giving us the
desired contradiction.

It remains to show that \eqref{eqn:sharp2} cannot hold.  This result
is known:  see, for instance, Wang~\cite{MR1127721}.  For
completeness, here we
construct a simple example of a function $f$ on the real line such
that $\|S_d f\|_p \geq cp^{1/2}\|f\|_p$.  Define the function $f$ on $\mathbb{R}$ by
\[ f(x) = \sum_{j=0}^\infty \chi_{(2^{-2j-1},2^{-2j})}(x). \]
Then
\[ \|f\|_p = \left(\sum_{j=0}^\infty 2^{-2j}-2^{-2j-1} \right)^{1/p} =
\left(\frac{1}{2}\sum_{j=0}^\infty 2^{-2j} \right)^{1/p} = \left(\frac{2}{3}\right)^{1/p}\leq 1. \]

To estimate the norm of $S_d f$, let $F_i=f_{Q_i}$, $i\ge 1$, denote the average of $f$ on the interval $Q_i=[0,2^{-i})$.  Then repeating the above calculation shows that
\[ F_{2i} = 2^{2i} \sum_{j=i}^\infty  2^{-2j}-2^{-2j-1} = \frac{2}{3}.  \]
Since the integrals of $f$ on $Q_{2i}$ and $Q_{2i-1}$ are the same, $F_{2i-1}=\frac{1}{3}$.   Therefore, given $i\ge 2$, if $2^{-2i-1} < x < 2^{-2i}$,
\[ S_df(x)^2 \geq \sum_{1\leq j \leq i} \big| F_{2j}-F_{2j-1}\big|^2  = \frac{i}{9} \geq c\log(1/x). \]
The same estimate (with a smaller constant $c$) holds when $2^{-2i}<x< 2^{-2i+1}$.  Therefore,
\[ \|S_df\|_p \geq c \left(\int_0^{2^{-3}} \log(1/x)^{p/2}\,dx \right)^{1/p}
\geq c \left( \sum_{k=3}^\infty k^{p/2}e^{-k}\right)^{1/p}
\geq c p^{1/2}, \]
where to get the last estimate we drop all the terms in the sum except for $k=[p]+2$.
Combining these two estimates, we see that
\[ \|S_df\|_p \geq c p^{1/2}\|f\|_p, \]
which is what we wanted to prove.

\subsection*{Two weight inequalities:  Proof of Theorem~\ref{thm:Bp-bump:Sd}}
Fix $p$, $1<p<\infty$.  Then, arguing as before it suffices to show
that
\[ \left(\int_{Q_N} |S_d f(x)^2-m_{(S_d f)^2}(Q_N)|^{p/2}
 u(x)\,dx\right)^{2/p} \leq C\,\|f\|_{L^p(v)}^2. \]
We again use \eqref{eqn:Sd-decomp}. To estimate  the term containing
$M$, note that we have
\[
\|(M f)^2\|_{L^{p/2}(Q_N,u)} \le \|M f\|_{L^{p}(u)}^2 \le C\,
\|f\|_{L^{p}(v)}^2,
\]
where we have used Theorem~\ref{thm:twowt-max} and the fact that
$(u,v)$ satisfies \eqref{eqn:Sd1} when $1<p\le 2$ or
(\ref{eqn:bp-bumpSqrt}) when $p>2$.

To estimate $F$ we consider two cases.  Suppose first that $1<p\leq
2$. Then  we use that $p/2\leq 1$, inequality \eqref{eqn:Sd1}, and
Theorem~\ref{thm:orlicz-max} and the fact that $\bar{B}\in B_p$ to
get
\begin{align*}
\int_{Q_N} F(x)^{p/2}\, u(x)\,dx
& \leq \sum_{j,k} \left(\avgint_{\widehat{Q}_j^k} |f(x)|\,dx\right)^p u(Q_j^k) \\
& \leq C\,\sum_{j,k} \avgint_{Q_j^k} u(x)\,dx \,
\|v^{-1/p}\|_{B,\widehat{Q}_j^k}^p \;
\|fv^{1/p}\|_{\bar{B},\widehat{Q}_j^k}^p |E_j^k| \\
& \leq C\sum_{j,k} \int_{E_j^k} M_{\bar{B}}(fv^{1/p})(x)^p \,dx \\
& \leq C\int_{\re^n} M_{\bar{B}}(fv^{1/p})(x)^p \,dx \\
& \leq C\|f\|_{L^p(v)}^p.
\end{align*}
Combining these two estimates we get the desired inequality.

\bigskip

Now suppose that $p>2$.  In this case the proof is very similar to
the proof of Theorem~\ref{thm:haarshift-twowt} and we highlight the
changes. To use duality with respect to Lebesgue measure,  fix
a non-negative function $h\in L^{(p/2)'}(\re^n)$ with  $\|h\|_{L^{(p/2)'}}=1$.
Then (\ref{eqn:bp-bumpSqrt}) gives
\begin{align*}
&\int_{Q_N} F(x)\,u(x)^{2/p}\,h(x)\,dx
\\
&\qquad \leq C\sum_{j,k} \left(\avgint_{\widehat{Q}_j^k}
|f(x)|\,dx\right)^{2} \avgint_{Q_j^k}
u(x)^{2/p}h(x)\,dx \; |E_j^k| \\
&\qquad \leq C\sum_{j,k}
\|fv^{1/p}\|_{\bar{B},\widehat{Q}_j^k}^2\|v^{-1/p}\|_{B,\widehat{Q}_j^k}^2
\|u^{2/p}\|_{A,Q_j^k}\|h\|_{\bar{A},Q_j^k}\;|E_j^k|\\
&\qquad \leq C\sum_{j,k} \int_{E_j^k}
M_{\bar{B}}(fv^{1/p})(x)^2 M_{\bar{A}} h(x)\,dx \\
&\qquad \leq C\int_{\re^n}
M_{\bar{B}}(fv^{1/p})(x)^2M_{\bar{A}} h(x)\,dx
\\
&\qquad \leq
C\,\|M_{\bar{B}}(fv^{1/p})(x)\|^2_{L^p}\,\|M_{\bar{A}}h \|_{L^{(p/2)'}}
\\
&\qquad \leq C\,\|f\|_{L^p(v)}^2,
\end{align*}
where we have used H\"older's inequality,
Theorem~\ref{thm:orlicz-max} and the fact that $\bar{A}\in
B_{(p/2)'}$ and $\bar{B}\in B_p$. The desired estimate follows at
once if we take the supremum over all such functions $h$.

\section{The vector-valued maximal operator}
\label{section:Mq}

Our two results for the vector-valued maximal operator are exact
parallels of our results for the dyadic square function.  Formally,
the change only requires replacing ``$2$'' by ``$q$'', $1<q<\infty$,
and in fact, the proofs do adapt readily as we will sketch below.

As with singular integral operators, in order to prove sharp results
for vector-valued maximal operator, we need to consider a dyadic
operator.  Recall that the dyadic maximal operator is defined by
\[ M^df(x) = \sup_{\substack{Q\in \Delta \\ Q\ni x}} \avgint_Q |f(y)|\,dy. \]
Given $q>1$ and $f=\{f_i\}$, define the dyadic vector valued maximal
operator by
\[ \overline{M}^d_q f(x) =  \left(\sum_{i=1}^\infty M^df_i(x)^q\right)^{1/q}.  \]
By an argument that goes back to C.~Fefferman and
Stein~\cite{fefferman-stein71} (see also \cite{sawyer82b} and
\cite{MR807149}), the maximal operator can be approximated by the
dyadic maximal operator and the analogous operator defined on all
translates of the dyadic grid.  Therefore, by a straightforward
argument using Fatou's lemma and Minkowski's inequality,  to prove
weighted norm inequalities for the vector-valued maximal operator it
suffices to prove them for $\overline{M}_q^d$.  (For the details of
this argument, see~\cite{cruz-martell-perezBook}.)

Again like the dyadic square function, the key estimate for the
dyadic vector-valued maximal operator is to control the local mean
oscillation of $(\overline{M}_q^d f)^q$.

\newcommand{\Md}{\overline{M}_q^d}

\begin{lemma} \label{lemma:osc-est-Mq}
Fix $\lambda$, $0<\lambda <1$, and $q$, $1<q<\infty$.  Then
for any function $f=\{f_i\}$, every dyadic cube $Q_0$, and every
$x\in Q_0$,
\begin{gather*}
\omega_\lambda ((\Md f)^q,Q_0) \leq
\frac{C_{n,q}}{\lambda^q}\left(\avgint_{Q_0}\|f(x)\|_{\ell^q}\,dx\right)^q, \\
M^{\#,d}_{\lambda,Q_0}((\Md f)^q)(x) \leq \frac{C_{n,q}}{\lambda^q}
M^d(\|f(\cdot)\|_{\ell^q} )(x)^q.
\end{gather*}
\end{lemma}

\begin{proof}
The second estimate again follows from the first. To prove the
first, fix $Q_0$.  Then for every $x\in Q_0$ and every $i\geq 1$, we
observe that
\[ M^df_i(x) = \max\left( M^d(f_i\chi_{Q_0})(x), \sup_{\substack{Q\in\Delta\\ Q_0\subset Q}}\avgint_Q |f_i(y)|\,dy\right). \]
The second term on the right is constant; using this we define
\[ K_0 = \left(\sum_{i=1}^\infty \left(\sup_{\substack{Q\in\Delta\\ Q_0\subset Q}}\avgint_Q |f_i(y)|\,dy\right)^q\right)^{1/q}. \]
For $x\in Q_0$, $\Md f(x)^q \geq K_0^q$.  We also have the following
elementary inequality: for every $a,b\ge 0$, $0\le \max(a,b)-b \leq
a$.  Combining these facts we get that
\[ 0
 \leq \Md f(x)^q -  K_0^q
\leq \sum_{i=1}^\infty M^d( f_i\chi_{Q_0})(x)^q =
\Md(f\chi_{Q_0})(x)^q. \]
Since the vector-valued maximal operator is weak $(1,1)$ (see
\cite{fefferman-stein71}), for any $t>0$,
\begin{multline*}
 |\{ x\in Q_0 : |\Md f(x)^q - K_0^q| > t \}| \\
\leq |\{ x\in Q_0 : \Md (f\chi_{Q_0})(x) > t^{1/q} \}| \leq
\frac{C_{n,q}}{t^{1/q}} \int_{Q_0} \|f(x)\|_{\ell^q}\,dx.
\end{multline*}
Therefore,  by \eqref{eqn:mean-est1} with $p=1/q$,
\begin{multline*}
\omega_\lambda ((\Md f)^q, Q_0)  \\
\leq \big(((\Md f)^q- K_0^q)\chi_{Q_0}\big)^*(\lambda |Q_0|) \leq
\frac{C_{n,q}}{\lambda^q} \left(\avgint_{Q_0}
\|f(x)\|_{\ell^q}\,dx\right)^q.
\end{multline*}
\end{proof}

\subsection*{One weight inequalities:  Proof of Theorem~\ref{thm:vvmax-onewt}}

As we noted above, the proof is very similar to the proof of
Theorem~\ref{thm:square}, and so we briefly sketch the key details.
By Theorem~\ref{thm:sharp-extrapol} it will suffice to prove it for
the special case when $p=q+1$. For this value of $p$ we have that
$(p/q)'=p$ and $1-p'=-1/q$.  As before, fix $w\in A_p$ and $Q_N$; we
will show that
\begin{multline*}
\left(\int_{Q_N} |\Md f(x)^q - m_{(\Md f)^q}(Q_N)|^{p/q} w(x)\,dx
\right)^{q/p}
\\
\leq C_{n,q}[w]_{A_{p}} \left(\int_{\re^n} \|f(x)\|_{\ell^q}^{p}
w(x)\,dx\right)^{q/p}.
\end{multline*}
By Theorem~\ref{thm:lerner}  and Lemma~\ref{lemma:osc-est-Mq}, for
every $x\in Q_N$,
\begin{align}\label{eqn:Mq-decomp}
&|\Md f(x)^q - m_{(\Md f)^q}(Q_N)|
\\ \nonumber
& \ \ \le C_{n,q}\,M(\|f(\cdot)\|_{\ell^q})(x)^q + C_{n,q}
\sum_{j,k} \left(\avgint_{\widehat{Q}_j^k}
\|f(x)\|_{\ell^q}\,dx\right)^q\chi_{Q_j^k}(x)
\\ \nonumber
&\ \ = C_{n,q}\,M(\|f(\cdot)\|_{\ell^q})(x)^q+C_{n,q}\,F(x).
\end{align}
To estimate the first term we use Theorem \ref{thm:sharpmax}. The
estimate for $F$ uses duality: fix a non-negative function $h\in L^p(w)$ with
$\|h\|_{L^p(w)}$ =1 (recall that $(p/q)'=p$). Then, proceeding as
before, we use the definition of $A_p=A_{q+1}$ to show that
\begin{align*}
&\int_{Q_N} F(x)\,h(x)\,w(x)\,dx
\\
&\quad \le C_{n}[w]_{A_p} \int_{\re^n}
M^d_{w^{-1/q}}(\|f(\cdot)\|_{\ell^q}w^{1/q})(x)^q w(x)^{-1/p}
M_w^dh(x)w(x)^{1/p}dx.
\end{align*}
Finally, we use H\"older's inequality, Theorem~\ref{lemma:wtdmax} and
then take the supremum over all such functions $h$ to get the
desired estimate.

To prove that the exponent
$\max\left(\frac{1}{q},\frac{1}{p-1}\right)$ is the best possible, we
consider two cases.  If $p\leq q+1$, then the exponent is
$\frac{1}{p-1}$, which is the same as the sharp exponent for the
scalar maximal function.  Therefore, the examples given by
Buckley~\cite{MR1124164} immediately adapt to the vector-valued
maximal operator.

If $p>q+1$, then we can argue exactly as we did for the dyadic square
function, replacing the exponent $1/2$ by $1/q$.  Therefore, to show
that the exponent $1/q$ is sharp we need to show that there exists a
vector-valued function $f=\{f_i\}$ such that $\|\overline{M}_q f\|_p
\geq cp^{1/q}\|f\|_p$.  But such a function is given by
Stein~\cite[p.~75]{MR1232192}.

\subsection*{Two weight inequalities:  Proof of Theorem~\ref{thm:vvmax-twowt}}
The proof is again nearly the same as the proof of
Theorem~\ref{thm:Bp-bump:Sd} for the dyadic square function, so we
only sketch the highlights.  Fix $p$, $1<p<\infty$; then it suffices
to show that
\[ \int_{Q_N} |\Md f(x)^q - m_{(\Md f)^q}(Q_n)|^{p/q} u(x)\,dx
\le C\,\int_{\re^n} \|f(x)\|_{\ell^q}^p\,v(x)\,dx.
\]
We use \eqref{eqn:Mq-decomp}. We estimate the term
involving $M$ using Theorem~\ref{thm:twowt-max} and the fact that
$(u,v)$ satisfies \eqref{eqn:vvmax1} when $1<p\le q$ or
\eqref{eqn:vvmax3} when $p>q$.

To estimate $F$ we consider two cases.  Suppose first that $1<p\leq
q$, then
$$
\int_{Q_N} F(x)^{p/q}\, u(x)\,dx
 \leq \sum_{j,k} \left(\avgint_{\widehat{Q}_j^k} \|f(x) \|_{\ell^q}\,dx\right)^p u(Q_j^k),
$$
and this term is estimated exactly as before. Combining these two
estimates we get the desired inequality.

\bigskip

When $p>q$,  we use duality with respect to Lebesgue measure and consider a
non-negative function $h$ such that $\|h\|_{L^{(p/q)'}}=1$. Then,
\begin{align*}
&\int_{Q_N} F(x)\,u(x)^{q/p}\,h(x)\,dx
\\
&\qquad\leq C_{n,q}\,\sum_{j,k} \left(\avgint_{\widehat{Q}_j^k}
\|f(x)\|_{\ell^q}\,dx\right)^{q} \avgint_{Q_j^k}
u(x)^{q/p}h(x)\,dx \; |E_j^k|.
\end{align*}
From here we follow the argument in the proof of
Theorem~\ref{thm:Bp-bump:Sd}, replacing $2$ by $q$.

\bibliographystyle{plain}
\bibliography{dyadic-hilbert}

\end{document}